\documentclass[reqno,12pt]{amsart}
\usepackage{bbm}
\usepackage[all,pdf]{xy}
\usepackage{epsfig}
\usepackage{amsmath,amssymb,amscd,mathtools}
\usepackage{mathrsfs}
\usepackage{graphicx}
\usepackage{color}
\usepackage{mathptmx}
\usepackage[top=30mm, bottom=30mm, left=30mm, right=30mm]{geometry}
\usepackage{aliascnt}
\usepackage[colorlinks=true,citecolor=blue,linkcolor=blue,urlcolor=blue]{hyperref}
\usepackage[nameinlink,capitalize]{cleveref}
\allowdisplaybreaks[4]
\makeatletter
\@namedef{subjclassname@2020}{%
	\textup{2020} Mathematics Subject Classification}
\makeatother

\newtheorem{theorem}{Theorem}[section]
\newaliascnt{proposition}{theorem}
\newtheorem{proposition}[proposition]{Proposition}
\aliascntresetthe{proposition}

\newaliascnt{lemma}{theorem}
\newtheorem{lemma}[lemma]{Lemma}
\aliascntresetthe{lemma}

\newaliascnt{corollary}{theorem}
\newtheorem{corollary}[corollary]{Corollary}
\aliascntresetthe{corollary}

\newaliascnt{question}{theorem}
\newtheorem{question}[question]{Question}
\aliascntresetthe{question}

\newaliascnt{claim}{theorem}

\aliascntresetthe{claim}

\theoremstyle{definition}
\newaliascnt{definition}{theorem}
\newtheorem{definition}[definition]{Definition}
\aliascntresetthe{definition}

\crefname{theorem}{Theorem}{Theorems}
\Crefname{theorem}{Theorem}{Theorems}
\crefname{proposition}{Proposition}{Propositions}
\Crefname{proposition}{Proposition}{Propositions}
\crefname{lemma}{Lemma}{Lemmas}
\Crefname{lemma}{Lemma}{Lemmas}
\crefname{corollary}{Corollary}{Corollaries}
\Crefname{corollary}{Corollary}{Corollaries}
\crefname{question}{Question}{Questions}
\Crefname{question}{Question}{Questions}
\crefname{claim}{Claim}{Claims}
\Crefname{claim}{Claim}{Claims}
\crefname{definition}{Definition}{Definitions}
\Crefname{definition}{Definition}{Definitions}
\crefname{section}{Section}{Sections}
\Crefname{section}{Section}{Sections}
\crefname{equation}{equation}{equations}
\Crefname{equation}{Equation}{Equations}

\numberwithin{equation}{section}

\DeclareMathOperator{\E}{\mathbb E}
\newcommand{\N}{\mathbb N}
\newcommand{\Z}{\mathbb Z}
\newcommand{\R}{\mathbb R}

\newcommand{\cX}{\mathcal X}
\providecommand{\X}{\mathcal X}

\newcommand{\cI}{\mathcal I}

\newcommand{\cP}{\mathcal P}
\newcommand{\cY}{\mathcal Y}

\newcommand{\bd}{d^*}
\newcommand{\one}{\mathbf 1}
\newcommand{\norm}[1]{\left\lVert #1\right\rVert}

\title[]{A three-dimensional corner configuration involving the Omega function}
\author[]{Zhuowen Guo, Rongzhong Xiao, and Shuhao Zhang}

\address[Zhuowen Guo, Rongzhong Xiao, and Shuhao Zhang]{School of Mathematical Sciences, University of Science and Technology of China, Hefei, Anhui 230026, P.R. China}
\email{guozw0920@mail.ustc.edu.cn}

\email{xiaorz@mail.ustc.edu.cn}

\email{yichen12@mail.ustc.edu.cn}

\date{\today}

\begin{document}
	
	\begin{abstract}
		Let \(\Omega(n)\) denote the number of prime factors of \(n\), counted with
		multiplicity.  We prove that if \(A\subset\mathbb N^3\) has positive upper
		Banach density, then there are $(x,y,z)\in \N^3, d\in \N$ such that
		$$(x,y,z),(x+d,y,z),(x,y+d,z),(x,y,z+\Omega(d))\in A.$$
		To establish the above result, we give an \(L^2\)-decoupling theorem for the triple ergodic averages
		\[
		\frac1N\sum_{n=1}^N T_1^n f_1\,T_2^n f_2\,S^{\Omega(n)}g
		\]
		associated with three commuting transformations by isotropy factors and nilpotent structures in $\Z^{2}$-actions.
	\end{abstract}
	\maketitle
	\section{Introduction}
	\subsection{Background and motivation}
	Let \(\Omega(n)\) be the number of prime factors of \(n\), counted with
	multiplicity.  Bergelson and Richter \cite{BergelsonRichter} developed a dynamical approach to this
	function by proving that, in every uniquely ergodic system, any orbit sampled
	at the times \(\Omega(n)\) equidistributes with continuous observations according to the invariant measure.  
	More recently, Bergelson, Reilly, and Richter \cite{BergelsonReillyRichter2026} developed a common
	weighted distribution framework for a broad class of arithmetic functions
	satisfying a Gaussian distribution condition; the class includes
	\(\Omega\), \(\omega\), squarefree indexed prime factor counts, and digit sum
	functions, and yields a corresponding uniquely ergodic orbit theorem.  
	For various variants of the above results, one can see Loyd \cite{Loyd2023}, Charamaras \cite{Charamaras2025}, Loyd-Mondal \cite{LoydMondal2025}, Wang-Wei-Yan-Yi \cite{WangWeiYanYi2025}, Xiao \cite{Xiao2026}, C\' espedes-Donoso \cite{CespedesDonoso2026} and so on
	for more details.
	
	The present paper concerns additive configurations in which two directions
	have a common linear step and a third direction has step \(\Omega(d)\).  For
	\(A\subset\mathbb Z^k\), recall that
	\[
	d^*(A):=\sup_{\Phi}\limsup_{N\to\infty}
	\frac{|A\cap\Phi_N|}{|\Phi_N|},
	\]
	where the supremum is taken over all F\o lner sequences
	\(\Phi=(\Phi_N)_{N\geq1}\) in \(\mathbb Z^k\).  We say that \(A\) has
	positive upper Banach density if \(d^*(A)>0\).  
	
	In 2025, Charamaras proved the following Roth-type result:
	\begin{theorem}[Charamaras, {\cite[Corollary~1.37]{Charamaras2025}}]
		\label{thm1-1}
		If \(E\subset\mathbb N\) has positive upper Banach density, then there are
		\(m,n\in\mathbb N\) such that
		\[
		m,m+n,m+\Omega(n)\in E.
		\]
	\end{theorem}
	
	The second author subsequently obtained a multi-dimensional polynomial version of \Cref{thm1-1}.
	\begin{theorem}[Xiao, {\cite[Proposition~1.11]{Xiao2026}}]
		\label{thm1-2}
		Let \(P_1,\ldots,P_k\in\mathbb Z[n]\) have zero constant terms and be
		pairwise independent, meaning that no non-trivial integer linear combination
		of any $P_i,P_j(i\neq j)$ is constant.  Assume also that at least one \(P_i\)
		is not of the form \(cn^r\).  If \(A\subset\mathbb N^{k+1}\) has positive
		upper Banach density, then there are \(a\in\mathbb N^{k+1}\) and
		\(d\in\mathbb N\) such that
		\[
		a,a+P_1(d)e_1,\ldots,a+P_k(d)e_k,
		a+\Omega(d)e_{k+1}\in A,
		\]
		where $e_i,1\le i\le k+1$ is the standard basis of $\R^{k+1}$.
	\end{theorem}
	
	The non-degeneracy assumption in \Cref{thm1-2} excludes the common linear
	family \(P_1=\cdots=P_k=n\).  This led to the following question, which motivates the paper.
	\begin{question}[Xiao, {\cite[Question~7.1]{Xiao2026}}]
		\label{ques1-1}
		Fix \(k\geq2\).  If \(A\subset\mathbb N^{k+1}\) has positive upper Banach
		density, must there exist \(a\in\mathbb N^{k+1}\) and \(d\in\mathbb N\)
		such that
		\[
		a,a+d e_1,\ldots,a+d e_k,
		a+\Omega(d)e_{k+1}\in A?
		\]
	\end{question}
	\subsection{Main results}
	In this paper, we answer \Cref{ques1-1} for \(k=2\) positively by the following.
	\begin{theorem}\label{thm:main-corners}
		For every \(A\subset\mathbb N^3\) with \(d^*(A)>0\), there are
		\(a\in\mathbb N^3\) and \(d\in\mathbb N\) such that
		\[
		a,a+d e_1,a+d e_2,
		a+\Omega(d)e_3\in A.
		\]
	\end{theorem}
	\begin{proof}
		See \Cref{cor:fixed-directions}.
	\end{proof}
	
	Conventionally, by Furstenberg's correspondence principle, the above additive configurations can be derived from the following decoupling result, which is the main ergodic theoretic conclusion of the paper.
	\begin{theorem}\label{thm:main-decoupling}
		Let \(T_1,T_2,S\) be three commuting invertible measure-preserving transformations
		of a standard probability space \((X,\mathcal X,\mu)\).  Then, for all
		\(f_1,f_2,g\in L^\infty(\mu)\),
		\begin{equation}\label{eq:main-decoupling}
			\lim_{N\to\infty}
			\left\|
			\frac1N\sum_{n=1}^N
			T_1^n f_1\,T_2^n f_2\,
			S^{\Omega(n)}\bigl(g-\E_{\mu}(g\mid\mathcal I(S))\bigr)
			\right\|_{L^2(\mu)}=0,
		\end{equation}
		where $\mathcal I(S)$ is the sub-$\sigma$-algebra of $\X$ generated by all $S$-invariant sets.
	\end{theorem}
	\begin{proof}
		\Cref{prop:prime-dilation-austin-reduction} and \Cref{prop:structured-decoupling} give it.
	\end{proof}
	Note that by Tao \cite{Tao}, the unweighted double commuting averages
	$$\frac{1}{N}\sum_{n\leq N}T_1^n f_1T_2^n f_2$$ converge in \(L^2\)-norm. Therefore, \Cref{thm:main-decoupling} identifies the existence of \(L^2\)-limit of
	the triple ergodic average
	\[
	\frac1N\sum_{n=1}^N T_1^n f_1\,T_2^n f_2\,S^{\Omega(n)}g.
	\]
	When \(T_1=T_2\), our decoupling result follows from
	\cite[Equation~(5.9)]{Xiao2026}; under the ergodicity
	assumption, it is also contained in \cite[Corollary~1.33]{Charamaras2025}.
	
	Next, we introduce some natural extensions or variants of \Cref{thm:main-corners} and \Cref{thm:main-decoupling}. The first one is a weighted form of \Cref{thm:main-decoupling}. Recall that a bounded sequence \(b:\mathbb Z\to\mathbb C\) is
	\emph{Besicovitch almost periodic} if, for every \(\varepsilon>0\), there is a
	trigonometric polynomial $P(n)$
	such that
	\[
	\limsup_{N\to\infty}\frac1N\sum_{n=1}^N|b(n)-P(n)|^2
	<\varepsilon^2.
	\]
	\begin{corollary}
		\label{cor:fixed-powers-coefficients}
		Let \(b:\mathbb Z\to\mathbb C\) be bounded and Besicovitch almost periodic.  Let \(T_1,T_2,S\) be three commuting invertible measure-preserving transformations
		of a standard probability space \((X,\mathcal X,\mu)\).  Then, for all
		\(f_1,f_2,g\in L^\infty(\mu)\),
		\[
		\frac1N\sum_{n=1}^N b(n)\,T_1^{n}f_1\,T_2^{n}f_2\,
		S^{\Omega(n)}
		\bigl(g-\E_{\mu}(g\mid\mathcal I(S))\bigr)
		\longrightarrow0
		\]
		in \(L^2(\mu)\) as $N\to\infty$.
	\end{corollary}
	\begin{proof}
		By the definition of bounded Besicovitch almost periodic sequence and the linear property of ergodic averages, it suffices to
		treat the case where \(b(n)=e^{2\pi i n\alpha}\).  On
		\(X\times(\mathbb R/\mathbb Z)\), 
		we define
		$
		\widetilde T_1(x,t)=(T_1x,t+\alpha),
		\widetilde T_2(x,t)=(T_2x,t),$ and $
		\widetilde S(x,t)=(Sx,t).
		$ Let $m$ be the Haar measure on \(\mathbb R/\mathbb Z\). Then $\widetilde T_1,\widetilde T_2,\widetilde S$ are three commuting invertible measure-preserving transformations
		of \((X\times (\mathbb R/\mathbb Z),\mathcal X\otimes \mathcal B(\mathbb R/\mathbb Z),\mu\times m)\).
		Applying \Cref{thm:main-decoupling} to  $\widetilde T_1,\widetilde T_2,\widetilde S$ and
		$
		\widetilde f_1(x,t):=e^{2\pi i t}f_1(x),
		\widetilde f_2(x,t):=f_2(x),
		\widetilde g(x,t):=g(x)
		$ and the fact that 
		$
		\mathcal I(\widetilde S)
		=\mathcal I(S)\otimes\mathcal B(\mathbb R/\mathbb Z)
		$ gives the Corollary.
	\end{proof}
	The second one is the general form of \Cref{thm:main-corners}. Its proof is presented in \Cref{sec:corners}.
	\begin{corollary}
		\label{cor:fixed-directions}
		Given \(\ell\geq 1\), let \(u_1,u_2,u_3\in\mathbb Z^\ell\). Assume that they are non-zero. For any 
		\(A\subset\mathbb Z^\ell\) with positive upper Banach density, the set
		of \(d\in\mathbb N\) for which
		\[
		d^*\bigl(A\cap(A-du_1)\cap(A-du_2)
		\cap(A-\Omega(d)u_3)\bigr)>0
		\]
		has positive upper Banach density.
	\end{corollary}
	The third one demonstrates  arbitrarily prescribed divisibility of the common differences in \Cref{cor:fixed-directions}.
	\begin{corollary}
		\label{cor:divisible-common-differences}
		Given \(\ell\geq 1\), let \(u_1,u_2,u_3\in\mathbb Z^\ell\). Assume that they are non-zero. For any 
		\(A\subset\mathbb Z^\ell\) with positive upper Banach density and all
		\(Q\in\mathbb N\), the set of \(d\in Q\mathbb N\) for which
		\[
		d^*\bigl(A\cap(A-du_1)\cap(A-du_2)
		\cap(A-\Omega(d)u_3)\bigr)>0
		\]
		has positive upper Banach density.
	\end{corollary}
	\begin{proof}
		It is easy to see that $\one_{Q\mid n}$ is a bounded Besicovitch almost periodic sequence. The rest proof is the repeats of the proof of \Cref{cor:fixed-directions} except a minor adjustment that we use \Cref{cor:fixed-powers-coefficients} instead of \Cref{thm:main-decoupling}.
	\end{proof}

	\subsection{A brief outline of the proof of \Cref{thm:main-decoupling}}
	Put \(T^{(a,b)}=T_1^aT_2^b\), \(e_1=(1,0)\), and \(e_2=(0,1)\).
	The proof consists of a characteristic factor reduction followed by a
	one-action decoupling argument for the resulting structured functions.
	
	\smallskip
	\noindent\emph{Characteristic factor reduction.}
	For distinct primes \(p,q\), the identity
	\(\Omega(pn)=\Omega(qn)=\Omega(n)+1\) cancels the \(S\)-coordinate in
	the prime dilated correlations. Thus
	\Cref{lem:prime-dilation-to-three-direction} reduces the required
	vanishing to characteristic factors for certain three-direction linear
	averages of the \(\mathbb Z^2\)-action \(T\).  Austin's pleasant extension
	theorem \Cref{thm:austin-pleasant}, together with
	\Cref{prop:austin-pronil-domination,lem:S-preserving-lift}, shows that
	these factors are contained in joins of a \(T\)-invariant subfactor
	\(\mathcal N_X\subseteq Z_2(T)\) and countably many isotropy factors
	\(\cI(T^w)\).  This gives the first reduction,
	\Cref{prop:prime-dilation-austin-reduction}.
	
	\smallskip
	\noindent\emph{Finite structured functions.}
	The Martingale Theorem reduces functions
	measurable with respect to these joins to finite sums of terms of the form
	$
	H_i\prod_{w\in W_i}h_{i,w},
	H_i\in L^\infty(X,\mathcal N_X,\mu),
	h_{i,w}\in L^\infty(X,\cI(T^w),\mu).
	$
	Since \(\mathcal N_X\subseteq Z_2(T)\),
	\Cref{lem:smooth-bundle-approx} approximates the finitely many functions
	\(H_i\) by tame $2$-step nilfunctions.  It therefore remains to prove the
	claim for the elementary finite structured functions of
	\Cref{def:finite-structured-algebra}.
	
	\smallskip
	\noindent\emph{Reduction to one action.}
	Choose \(v\in\mathbb Z^2\) transverse to all of the finitely many isotropy
	directions and set \(U=T^v\).  On each residue class modulo
	\(\lvert\det(v,w)\rvert\), \Cref{lem:isotropy-coding} rewrites every term
	\(T^{ne_i}h_{i,w}\) as an affine iterate of \(U\).  At the same time,
	\Cref{lem:bundle-weight-ezk-stable} shows that the product of the translated
	tame nilfunctions is an Eisner--Zorin-Kranich regular weight for \(U\).
	After passing to a common residue modulus, the cyclic tower construction
	turns all remaining terms into affine iterates of a single transformation.
	\Cref{prop:cyclic-tower-one-action-decoupling} then separates
	\(S^{\Omega(n)}g\) from the other factors, leaving precisely
	\(\E_\mu(g\mid\cI(S))\).  In its proof, the orthogonal part is
	handled by \Cref{prop:multislope-nilsequence-estimate}, while the structured
	part is handled by tame nilfunction approximation and \Cref{lem:xiao-nilsequence-family}, based on the
	Bergelson--Richter's ergodic theorem \Cref{thm:beg-omega-nilsequence}.
	This proves the finite structured result
	\Cref{prop:finite-nil-decoupling}.
	
	Finally, the Martingale Theorem and tame nilfunction
	approximations pass this conclusion to the full characteristic factors,
	giving \Cref{prop:structured-decoupling}.  Combining
	\Cref{prop:structured-decoupling} with
	\Cref{prop:prime-dilation-austin-reduction} proves
	\Cref{thm:main-decoupling}. 
	
	\subsection*{Organization of the paper}
	\Cref{sec:preliminaries} collects some notions, notations, and results.  \Cref{sec:corners} proves \Cref{cor:fixed-directions} assuming \Cref{thm:main-decoupling}. \Cref{sec:orthogonality-austin} gives the first reduction of \Cref{thm:main-decoupling} by determining the related characteristic factors. \Cref{sec:nil-one-action} contain an approximation by tame nilfunctions and norm convergence results of two classes of weighted multiple one-action ergodic averages. \Cref{Sec6} focuses on cyclic tower action and its properties. \Cref{sec:isotropy-completion} gives the second reduction of \Cref{thm:main-decoupling} by elementary finite structured functions. \Cref{sec:further-questions} contains some further discussions and directions. \Cref{app:covering-presentation} lifts linear orbit on nilmanifolds to polynomial orbits on ``admissible" nilmanifolds. \Cref{appB} focuses on Eisner--Zorin-Kranich regular weights.
	\subsection*{Acknowledgements}
	The authors would like to thank Professors Jiahao Qiu and Leiye Xu, for many helpful discussions and valuable suggestions. Rongzhong Xiao is supported by the National Natural Science Foundation of China (No. 123B2007, 12426201) and the Postdoctoral Fellowship Program of CPSF under Grant Number GZB20260746. Shuhao Zhang is supported by National Key R\&D Program of China (No. 2024YFA1013602, 2024YFA1013600) and NNSF of China (12031019, 12371197, 12426201).
	\section{Preliminaries}
	\label{sec:preliminaries}
	
	\subsection{Actions, factors, inverse limits, and characteristic factors}
	\label{subsec:actions-factors-notation}
	
	Let \(G\) be a group.  A \emph{measure-preserving \(G\)-system} is a standard
	probability space \((X,\mathcal X,\mu)\) together with a group homomorphism
	\(T:G\to\operatorname{Aut}(X,\mu)\), where $\operatorname{Aut}(X,\mu)$ is the group of invertible measure-preserving transformations of \((X,\mathcal X,\mu)\). The image of
	\(g\in G\) is written \(T^g\).  Thus \(T^{gh}=T^g\circ T^h\) for all
	\(g,h\in G\).  Such a system is denoted by \((X,\mathcal X,\mu,(T^g)_{g\in G})\), or simply
	\((X,\mu,(T^g)_{g\in G})\) when the \(\sigma\)-algebra is clear. For a $G$-system \((X,\mathcal X,\mu,(T^g)_{g\in G})\), we let $\mathcal{I}((T^g)_{g\in G})=\bigcap_{g\in G}\mathcal{I}(T^g)$, where $\mathcal{I}(T^g)$ is the the sub-$\sigma$-algebra of $\X$ generated by all $T^g$-invariant sets.
	If $\mathcal{I}((T^g)_{g\in G})$ is trivial, we say that \((X,\mathcal X,\mu,(T^g)_{g\in G})\) is {\em ergodic}. 
	
	{\bf All actions appearing in this paper are \(\mathbb Z^d\)-actions or restrictions
		to subgroups thereof.}  
	When $d=1$, we use $T$ to denote the map $T^{1}$ and use $(X,\X,\mu,T)$ to denote the $\Z$-system \((X,\mathcal X,\mu,(T^n)_{n\in \Z})\) . Then $\mathcal{I}(T^{1})=\mathcal{I}(T)=\mathcal{I}((T^n)_{n\in \Z})$. 
	%
	
	If \((Y,\mathcal Y,\nu,(S^g)_{g\in G})\) is another \(G\)-system, a measurable map
	$
	\pi:(Y,\mathcal Y,\nu,(S^g)_{g\in G})\to (X,\mathcal X,\mu,(T^g)_{g\in G})
	$
	is called a \emph{(measurable) factor map} if \(\pi_*\nu=\mu\) and
	\(\pi\circ S^g = T^g\circ\pi\) \(\nu\)-a.e. for every \(g\in G\). At this point, we say that $(X,\mathcal X,\mu,(T^g)_{g\in G})$ is a (measurable) factor of \((Y,\mathcal Y,\nu,(S^g)_{g\in G})\).or \((Y,\mathcal Y,\nu,(S^g)_{g\in G})\) is a (measurable) extension of $(X,\mathcal X,\mu,(T^g)_{g\in G})$.
	An \emph{isotropy factor} is a factor of the form \(\mathcal I(T^{w})\) for
	some non-zero \(w\in\mathbb Z^d\).
	
	Next, we introduce inverse limits. Let $I$ be a countable directed set. For each $i\in I$, assign a $G$-system $(X_i,\mathcal X_i,\mu_i,(T^g)_{g\in G})$ and for all $i,j\in I$ with $i\le j$, assign a factor map $\pi_{i,j}:X_j\to X_i$ such that for all $i\le j\le \ell\in I$, $\pi_{i,\ell}=\pi_{i,j}\circ \pi_{j,\ell}$. At this point, we say that the pair $((X_i,\mathcal X_i,\mu_i,(T^g)_{g\in G})_{i\in I},(\pi_{i,j})_{i,j\in I,i\le j})$ forms an {\em inverse system}. An {\em inverse limit} of this inverse system is defined abstractly to be a $G$-system $(X,\mathcal X,\mu,(T^g)_{g\in G})$ endowed with factor maps $\pi_{i}:X\to X_i$ satisfying (i) $\pi_{i}=\pi_{i,j}\circ \pi_{j}$ for all $i\le j\in I$ and the universal property: (ii) If $(Y,\mathcal Y,\nu,(T^g)_{g\in G})$ is a $G$-system and for $i\in I$, $p_i:Y\to X_i$ is a factor map such that $p_{i}=\pi_{i,j}\circ p_{j}$ for all $i\le j\in I$, then there is a unique factor map $p:Y\to X$ such that for all $i\in I$, $p_i=\pi_{i}\circ p$.
	
	In this paper, we also use characteristic factors. Fix a $G$-system $(X,\mathcal X,\mu,(T^g)_{g\in G})$ and sequences $a_1,\ldots,a_k:\Z\to G$. We say that the ergodic averages
	$$\frac{1}{N}\sum_{n=1}^{N}T^{a_1(n)}f_1\cdots T^{a_k(n)}f_k$$ admit the {\em characteristic factor} $\mathcal{A}$ in the $i$-th coordinates if for all $f_1,\ldots,f_k\in L^{\infty}(\mu)$ with $\E_{\mu}(f_i|\mathcal{A})=0$, then
	$$\lim_{N\to\infty}\frac{1}{N}\sum_{n=1}^{N}T^{a_1(n)}f_1\cdots T^{a_k(n)}f_k=0$$ in $L^{2}(\mu)$.
	\subsection{Disintegration of measures}
	\label{subsec:measure-disintegration}
	
	Let \((Y,\mathcal Y,\nu)\) be a standard probability space and \(\mathcal C\subset\mathcal Y\) a sub-\(\sigma\)-algebra.  
	A {\em disintegration of \(\nu\) over \(\mathcal C\)} is a family \(\{\nu_{y}\}_{y\in Y}\) of probability measures on \((Y,\mathcal Y)\) such that for every bounded measurable \(f:Y\to\mathbb C\) the map \(y\mapsto\int_Y f\,d\nu_{y}\) is \(\mathcal C\)-measurable,
	\[
	\int_Y f\,d\nu = \int_Y\Bigl(\int_Y f\,d\nu_{y}\Bigr)d\nu(y),
	\]
	and \(\E_\nu(f\mid\mathcal C)(y) = \int_Y f\,d\nu_{y}\) for \(\nu\)-a.e. \(y\).  
	For \(\nu\)-a.e. \(y\), the measure \(\nu_{y}\) is concentrated on the atom of \(\mathcal C\) containing \(y\).
	
	When \(\mathcal C\) is generated by a measurable map \(\pi:(Y,\mathcal Y)\to(\Omega,\mathcal B)\) with \(\Omega\) a standard Borel space, the disintegration is written
	\[
	\nu = \int_\Omega \nu_\omega\,d\beta(\omega),\qquad \beta = \pi_*\nu,
	\]
	where for \(\beta\)-a.e. \(\omega\), the probability measure \(\nu_\omega\) is supported on \(\pi^{-1}\{\omega\}\).
	
	For a \(G\)-system \((X,\mathcal X,\mu,(T^g)_{g\in G})\) and a \(G\)-invariant sub-\(\sigma\)-algebra \(\mathcal C\), the disintegration of \(\mu\) over \(\mathcal C\) gives conditional measures \(\mu_x\) for which 
	\[
	(T^g)_*\mu_x = \mu_{T^g x}\ \text{for }\mu\text{-a.e. }x,\ \text{and all}\ g\in G.
	\]
	If, in addition, \(\mathcal C\subseteq\mathcal{I}((T^g)_{g\in G})\), then \(\mu_x\) is \(G\)-invariant for \(\mu\)-a.e. \(x\).
	In particular, when \(\mathcal C = \mathcal{I}((T^g)_{g\in G})\), this is the \emph{ergodic decomposition} of $\mu$ (with respect to $T$) and the conditional measures are ergodic. 
	%
	\subsection{Host-Kra seminorms}
	\label{subsec:hk-cube-measures}
	
	Fix $(X,\cX,\mu,(T^v)_{v\in \Z^d})$.  For
	$j\geq0$, set
	$
	X^{[j]}=X^{\{0,1\}^j},
	\cX^{[j]}=\cX^{\otimes\{0,1\}^j},
	$
	with $X^{[0]}=X$ and $\cX^{[0]}=\cX$.  Define the {\em cube measures}
	$\mu_X^{[j]}$ inductively.  Put $\mu_X^{[0]}=\mu$.  If $\mu_X^{[j]}$ has
	been defined, let $T_{[j]}$ be the diagonal $\Z^d$-action on $X^{[j]}$ as follow:
	\[
	T_{[j]}^t\bigl((x_\omega)_{\omega\in\{0,1\}^j}\bigr)
	:=
	\bigl(T^t x_\omega\bigr)_{\omega\in\{0,1\}^j},
	t\in\Z^d.
	\]
	Let $\cI_X^{[j]}=\cI((T_{[j]}^t)_{t\in \Z^d})$.  Identify
	$X^{[j+1]}$ with $X^{[j]}\times X^{[j]}$ by the following map pair
	\[
	(x_\eta)_{\eta\in\{0,1\}^{j+1}}
	\longleftrightarrow
	\bigl((x_{\omega0})_{\omega\in\{0,1\}^j},
	(x_{\omega1})_{\omega\in\{0,1\}^j}\bigr).
	\]
	Then $\mu_X^{[j+1]}$ is the relatively independent self-joining of
	$\mu_X^{[j]}$ over $\cI_X^{[j]}$. That is, for all
	$F,G\in L^\infty(\mu_X^{[j]})$,
	\[
	\int_{X^{[j+1]}}F(x')G(x'')\,d\mu_X^{[j+1]}(x',x'')
	=
	\int_{X^{[j]}}
	\E_{\mu_X^{[j]}}(F\mid\cI_X^{[j]})\E_{\mu_X^{[j]}}(G\mid\cI_X^{[j]})\,d\mu_X^{[j]}.
	\]
	For $f\in L^\infty(\mu)$ and $j\geq1$, the {\em $j$-th Host-Kra seminorm} of $f$  
	\[
	\|f\|_{U^j(X,\X,\mu,(T^v)_{v\in \Z^d})}^{2^j}
	=
	\int_{X^{[j]}}
	\prod_{\omega\in\{0,1\}^j}\mathcal C^{|\omega|}f(x_\omega)
	\,d\mu_X^{[j]}((x_\omega)_{\omega\in\{0,1\}^j}),
	\]
	where $\mathcal C z=\overline z$ and
	$|\omega|=\omega_1+\cdots+\omega_j$. When the bottom system (space) is clear, we will omit it and  write $\|f\|_{U^j}$ ($\|f\|_{U^j}(T)$) simply.
	
	For each $s\geq0$, there is a factor $Z_s(X,\cX,\mu,(T^v)_{v\in \Z^d})$ characterized by
	\[
	\|f\|_{U^{s+1}}=0
	\quad\Longleftrightarrow\quad
	\E(f\mid Z_s(X,\cX,\mu,(T^v)_{v\in \Z^d}))=0
	\]
	for all bounded measurable $f$.
	Set
	$\displaystyle Z_\infty(X,\cX,\mu,(T^v)_{v\in \Z^d})=\bigvee_{s\ge0}Z_s(X,\cX,\mu,(T^v)_{v\in \Z^d})$. When the bottom space is clear, we will omit it and  write $Z_{s}(T)$ simply.
	
	Fix $\Lambda\le \Z^d$. For a subaction
	$T|_\Lambda$, the notation $\mu_{X,T|_\Lambda}^{[j]}$,
	$\|\cdot\|_{U^j(T|_\Lambda)}$, and $Z_s(T|_\Lambda)$ refers to the same
	construction applied to the restricted system.
	
	Next, we record the following two lemmas for later use. The first one may be known. Since we can not find a precise and direct citation, we restate it and provide a proof here.
	\begin{lemma}
		\label{lem:one-dimensional-power-comparison}
		Let 
		\((X,\mathcal X,\mu,U)\) be a $\Z$-system, \(a\in\mathbb Z\setminus\{0\}\), and
		$1$-bounded \(H\in L^\infty(\mu)\).  Then, for every \(m\ge2\),
		\[
		\|H\|_{U^m(U^a)}
		\le
		C_{a,m}\|H\|_{U^m(U)} .
		\]
		Moreover, for every \(r\ge1\),
		\[
		\limsup_{H_0\to\infty}\frac1{H_0}\sum_{h=1}^{H_0}
		\|U^{ah}H\cdot\overline H\|_{U^r(U)}^{2^r}
		\le
		|a|\|H\|_{U^{r+1}(U)}^{2^{r+1}} .
		\]
	\end{lemma}
	
	\begin{proof}
		The first assertion follows from \cite[Lemma~3.1]{FrantzikinakisKucaJoint}.
		For the second assertion, we may assume that \(a>0\) and set
		\(A_n=\|U^nH\cdot \overline H\|_{U^r(U)}^{2^r}\).  By
		\cite[Chapter~8, Proposition 16]{HostKraBook},
		\[
		\|H\|_{U^{r+1}(U)}^{2^{r+1}}
		=
		\lim_{M\to\infty}\frac1M\sum_{n=1}^M A_n .
		\]
		Since \(A_n\ge0\),
		\[
		\frac1N\sum_{h=1}^N A_{ah}
		\le
		a\frac1{aN}\sum_{n=1}^{aN}A_n .
		\]
		Taking limsup gives the claim for \(a>0\).  The case \(a<0\) follows from the
		case \(a>0\) applied to \(U^{-1}\). This finishes the proof.
	\end{proof}
	The second one is a local property of Host-Kra seminorms.
	\begin{lemma}[Host--Kra, {\cite[Chapter 8, Proposition 18]{HostKraBook}}]
		\label{lem:HK-seminorm-ergodic-decomposition}
		Let \((X,\mathcal X,\mu,T)\) be a $\Z$-system,
		and let
		\(
		\mu=\int_\Omega \mu_\omega\,d\lambda(\omega)
		\)
		be its ergodic decomposition with respect to \(T\).  Then, for every
		\(m\geq1\) and all \(F\in L^\infty(\mu)\),
		\[
		\|F\|_{U^m(X,\mu,T)}^{2^m}
		=
		\int_\Omega
		\|F\|_{U^m(X,\mu_\omega,T)}^{2^m}
		\,d\lambda(\omega).
		\]
	\end{lemma}
	\subsection{Nilsystems and nilsequences}
	\label{subsec:nilpotent-preliminaries}
	
	Here, we use the standard nilmanifold and nilsequence terminology from
	\cite{HostKra,GreenTaoNilorbits}; for polynomial sequences adapted to
	filtrations, see \cite{LeibmanPolynomialSequences,GreenTaoNilorbits}.
	
	A {\em \(k\)-step nilmanifold} is a compact homogeneous space \(G/\Gamma\), where
	\(G\) is a \(k\)-step nilpotent Lie group and \(\Gamma\) is a discrete
	co-compact subgroup of $G$.
	
	A \emph{prefiltration} \(G_\bullet\) of degree at most \(k\) is a
	decreasing sequence of closed Lie groups
	\[
	G_0\geq G_1\geq\cdots\geq G_{k+1}=\{e\},
	[G_i,G_j]\subseteq G_{i+j},
	\]
	with \(G_m=\{e\}\) for \(m>k\).  It is a \emph{filtration} on a
	nilpotent Lie group \(G\) if \(G_0=G_1=G\).  If \(G/\Gamma\) is fixed, a
	closed subgroup \(H\leq G\) is called {\em \(\Gamma\)-rational} if
	\(H\cap\Gamma\) is co-compact in \(H\).  A filtration \(G_\bullet\) is
	called {\em \(\Gamma\)-rational} if every \(G_i\) is \(\Gamma\)-rational.
	
	For \(g:\mathbb Z\to G_0\), use the standard discrete derivative convention, 
	\[
	D_hg(n):=g(n)^{-1}g(n+h).
	\]
	For a prefiltration \(G_\bullet\), write
	\(G_{\bullet+1}:=(G_{i+1})_{i\geq0}\), which is again a
	prefiltration.  Following \cite[Definition~2.10]{EisnerZorinKranich},
	polynomiality is defined recursively in the degree of a prefiltration:
	for the 
	prefiltration with \(G_0=\{e\}\), the only
	polynomial sequence is the constant identity sequence; otherwise a map
	\(g:\mathbb Z\to G_0\) is \(G_\bullet\)-polynomial if, for every
	\(h\in\mathbb Z\), the derivative \(D_hg\) is
	\(G_{\bullet+1}\)-polynomial.  The set of all such maps is denoted by
	\(P(\mathbb Z,G_\bullet)\).
	
	A {\em \(k\)-step \(\Lambda\)-nilsystem} is a nilmanifold \(G/\Gamma\)
	equipped with the Haar measure and a group homomorphism \(\rho:\Lambda\to G\); each
	\(v\in\Lambda\) acts by \(x\Gamma\mapsto\rho(v)x\Gamma\).  
	
	A {\em basic polynomial nilsequence} is a sequence \(n\mapsto F(g(n)\Gamma)\), where
	\(G/\Gamma\) is a nilmanifold equipped with a \(\Gamma\)-rational filtration
	\(G_\bullet\), \(g\in P(\mathbb Z,G_\bullet)\), and
	\(F\in C(G/\Gamma)\).  A {\em nilsequence} is a uniform limit of basic polynomial
	nilsequences. By \cite[Chapter 14, Corollary 16]{HostKraBook}, any nilsequence admits the Ces\`aro mean along any Følner sequence.
	\subsection{Eisner--Zorin-Kranich estimate}
	\label{subsec:ezk-estimate}
	
	We shall use \cite[Theorem~4.1]{EisnerZorinKranich}.  The statement in
	that paper is formulated using a presentation of the nilmanifold for which
	connected component of the identity of the acting nilpotent Lie group is simply
	connected (see \cite[The fourth paragraph of Section 1]{EisnerZorinKranich}). We will keep this convention in the quoted theorem.  
	
	Whenever a {\em \(C^K\)-norm} is used on a nilmanifold \(M\), we fix a
	smooth Riemannian metric $g$ on \(M\) and set
	\[
	\|\phi\|_{C^K(M)}
	=
	\sum_{j=0}^{K}\sup_{y\in M}|\nabla^j\phi(y)|_{g},
	\]
	where \(\nabla^j\) denotes the \(j\)-th covariant derivative.
	
	A filtered nilmanifold \((G/\Gamma,G_\bullet)\) is called
	\emph{Eisner--Zorin-Kranich admissible} if it satisfies the following: (1) the connect component of the identity  \(G^\circ\) is simply connected;
	(2) \(G_\bullet\) is {\(\Gamma\)-rational}; (3)
	a Mal'cev basis
	\(\mathcal M:=\{X_1,\dots,X_d\}\) for
	\(G^\circ/(\Gamma\cap G^\circ)\) adapted to
	\(G^{\circ}_{\bullet}\), has been fixed (see \cite[Definition 2.15]{EisnerZorinKranich}).  
	
	Fix an Eisner--Zorin-Kranich admissible filtered nilmanifold $(G/\Gamma,G_\bullet)$. Then we get a  Mal'cev basis $\{X_1,\dots,X_d\}$ for the Lie algebra of $G$. Identify $X_i$ with their extensions to right vector fields on $G/\Gamma$. For
	\(j\in\mathbb N\) and \(1\leq p<\infty\), the {\em Sobolev norm}
	is defined by 
	\[
	\|F\|_{W^{j,p}(G/\Gamma)}^p
	=
	\sum_{a=0}^{j}
	\sum_{b_1,\dots,b_a=1}^{d}
	\|X_{b_1}\cdots X_{b_a}F\|_{L^p(G/\Gamma)}^p .
	\]
	
	A point \(x\) in an ergodic $\Z$-system \((X,\mu,T)\) is called {\em fully generic for
		\(f\in L^\infty(\mu)\) with respect to the  Følner sequence \(\Phi\)} if it is generic, along
	\(\Phi\), for every function in the (separable) \(T\)-invariant subalgebra
	generated by \(f\).
	
	\begin{theorem}[Eisner--Zorin-Kranich, {\cite[Theorem~4.1]{EisnerZorinKranich}}]
		\label{thm:ezk-original}
		Let \((X,\mu,T)\) be an ergodic $\Z$-system.  Let
		\(f\in L^\infty(\mu)\), and let \(x\) be fully generic for \(f\) with respect
		to a tempered Følner sequence \(\Phi\) in \(\mathbb Z\).  Then, for every
		\(\ell\in\mathbb N\) and all \(\varepsilon>0\), there exists \(N_0\) such
		that the following holds for all \(N\geq N_0\).
		
		For all Eisner--Zorin-Kranich admissible filtered nilmanifold \((G/\Gamma,G_\bullet)\) of
		filtration length \(\ell\), all \(F\in C^\infty(G/\Gamma)\), and all
		\(g\in P(\mathbb Z,G_\bullet)\),
		\[
		\left|
		\frac1{|\Phi_N|}
		\sum_{n\in\Phi_N}
		f(T^n x)F(g(n)\Gamma)
		\right|
		\lesssim_{G/\Gamma, G_\bullet}
		\|F\|_{W^{\kappa,2^\ell}(G/\Gamma)}
		\bigl(\|f\|_{U^{\ell+1}(X,\mu,T)}+\varepsilon\bigr),
		\]
		where
		\[
		\kappa
		=
		\sum_{r=1}^{\ell}
		(d_r-d_{r+1})\binom{\ell}{r-1},
		d_r=\dim G_r .
		\]
		The implied constant depends only on \(G/\Gamma\), \(G_\bullet\), and on
		the (fixed) Mal'cev basis implicit in the definition of Eisner--Zorin-Kranich admissibility.
	\end{theorem}
	Next, we apply \Cref{thm:ezk-original} to finitely many nilmanifolds.
	\begin{corollary}
		\label{cor:ezk-finite-Ck-family}
		Fix \(s\geq1\).  Let
		$
		\mathscr N:=
		\{(G_i/\Gamma_i,G_{i,\bullet},\mathcal M_i):1\leq i\leq L\}
		$
		be a finite family of Eisner--Zorin-Kranich admissible filtered nilmanifolds of degree at
		most \(s\).  Put
		$$
		\kappa_i(s)
		=
		\sum_{r=1}^{s}
		(d_{i,r}-d_{i,r+1})\binom{s}{r-1},
		d_{i,r}=\dim G_{i,r},
		$$
		where the filtration is extended by \(G_{i,r}=\{e\}\) for \(r>s\).  Let
		\(K\geq\max_i\kappa_i(s)\) and \(B>0\).  Then there exists
		\(C>0\), depending only on \(\mathscr N,s,B\), and \(K\), such
		that the following holds.
		
		Let \((X,\mu,T)\) be an ergodic $\Z$-system,
		\(f\in L^\infty(\mu)\), and \(x\) be fully generic for \(f\) along
		\(\{1,\dots,N\}\).  If
		$$
		b(n)=\sum_{i=1}^{L}F_i(g_i(n)\Gamma_i),$$ where $ 
		g_i\in P(\mathbb Z,G_{i,\bullet}),
		F_i\in C^\infty(G_i/\Gamma_i)\ \text{with}\ 
		\|F_i\|_{C^{K}(G_i/\Gamma_i)}\leq B,1\leq i\leq L,$
		then
		\[
		\limsup_{N\to\infty}
		\left|
		\frac1N\sum_{n=1}^{N} f(T^n x)b(n)
		\right|
		\leq
		C\|f\|_{U^{s+1}(X,\mu,T)} .
		\]
	\end{corollary}
	
	\begin{proof}
		First pad every filtration of degree smaller than \(s\) by identity groups,
		so that all elements of $\mathscr N$ are regarded as filtered nilmanifolds of common length
		\(s\).  Apply \Cref{thm:ezk-original} with \(\ell=s\) to each term
		\(F_i(g_i(n)\Gamma_i)\).  The Sobolev order for the \(i\)-th term is
		\(\kappa_i(s)\).  Since \(K\geq\kappa_i(s)\), by the definitions of \(C^K\)-norm and Sobolev norm, there is $D_i>0$ such that 
		\[
		\|F_i\|_{W^{\kappa_i(s),2^s}(G_i/\Gamma_i)}
		\leq
		D_i\|F_i\|_{C^{K}(G_i/\Gamma_i)}
		\leq D_iB,
		\]
		where \(D_i\) depends only on the $i$-th Eisner--Zorin-Kranich admissible filtered nilmanifold and \(K,s\).  Summing the resulting estimates over
		\(i=1,\ldots,L\), and then letting \(\varepsilon\downarrow0\), finishes the
		proof.
	\end{proof}
	\subsection{Bundles of nilsystems}
	We need the bundle of nilmanifolds to deal with the non-ergodic cases. First, let us introduce the notion.
	\begin{definition}[{\cite[Definition~2.2]{JamneshanMachado2026}}]
		\label{def:jm-bundle-nilsystems}
		Let \((\Omega,\mathcal B,\beta)\) be a standard probability space.
		Choose:
		\begin{itemize}
			\item[(1)] a measurable partition \(\{\Omega_j\}_{j\ge0}\) of \(\Omega\);
			\item[(2)] compactly generated nilpotent Lie groups $N_j$ with discrete co-compact
			subgroups \(\Delta_j\leq N_j\) for all $j\ge 0$;
			\item[(3)]  measurable maps
			$
			\varphi_j:\Omega_j\to \operatorname{Hom}(\mathbb Z^r,N_j)
			$
			such that for every \(j\ge0\) and \(\beta\)-a.e. \(\omega\in\Omega_j\), the
			\(\mathbb Z^r\)-action, induced by
			\(\varphi_j(\omega)\), on \((N_j/\Delta_j,m_{N_j/\Delta_j})\)  is ergodic.
		\end{itemize}
		The space 
		$
		\displaystyle Y=\bigsqcup_{j\ge0}\Omega_j\times N_j/\Delta_j$ equipped with the probability measure
		$\displaystyle\nu=\sum_{j\ge0}\beta|_{\Omega_j}\otimes m_{N_j/\Delta_j},
		$
		and the action defined for all $j\ge 0,t\in \Z^r$ and $(w,y)\in \Omega_j\times N_j/\Delta_j$ by
		$$
		t\cdot(\omega,y)=(\omega,\varphi_j(\omega)(t)\cdot y)$$
		is called {\em a bundle of nilsystems with (the base space \(\Omega\))}, and is denoted by
		$
		\displaystyle \bigsqcup_{j\ge0}\Omega_j\ltimes_{\varphi_j} N_j/\Delta_j .
		$
		It has step at most \(k\) if every \(N_j\) is nilpotent of step at most
		\(k\).  The projection \(Y\to\Omega\) is called the {\em bundle projection}.
	\end{definition}
	
	Fix $\Lambda\le \Z^r$. A {\em \(k\)-step bundle nilfactor} of a \(\Lambda\)-system is a factor that is
	isomorphic to a bundle of \(k\)-step \(\Lambda\)-nilsystems. 
	A {\em \(k\)-step bundle pronilfactor} is an
	inverse limit of \(k\)-step bundle nilfactors.
	%
	
	Next, we introduce a structure theorem.
	\begin{theorem}[Jamneshan--Machado, {\cite[Theorem~1.2]{JamneshanMachado2026}}]
		\label{thm:jm-structure}
		Fix a not necessarily ergodic \((X,\X,\mu,(T^v)_{v\in \Z^r})\). Let $k\ge 1$. Then there exists an inverse
		system of factors
		\[
		Y_1\longleftarrow Y_2\longleftarrow\cdots\longleftarrow X
		\]
		such that:
		\begin{enumerate}
			\item each \(Y_n\) is a bundle of \(k\)-step nilsystems over the
			invariant factor \(\Omega\), and the bundle projection \(Y_n\to\Omega\)
			agrees with the projection on the invariant factor;
			\item \(Z_k(X,\X,\mu,(T^v)_{v\in \Z^r})\) is (measure-theoretically) isomorphic to the inverse
			limit of the \(Y_n\).
		\end{enumerate}
	\end{theorem}
	
		
		In the rest of this subsection, we introduce the notion of tame
		nilfunctions as follows.
		\begin{definition}
			\label{def:tame-finite-stage-nilfunction}
			Fix \((X,\mathcal X,\mu,(T^v)_{v\in \Z^d})\). Let 
			\(\Lambda\leq\mathbb Z^d\), and let \(k\geq1\).  A function
			\(\phi\in L^\infty(\mu)\) is called a \emph{tame \(k\)-step
				nilfunction for \(T|_\Lambda\)} if there are a factor $Y$, which is  a bundle of \(k\)-step
			\(\Lambda\)-nilsystems, associated with the factor map
			\[
			\pi:X\to Y:=\bigsqcup_{j\ge0}B_j\times M_j,
			M_j=N_j/\Delta_j,
			\]
			and \(\Phi\in L^\infty(Y)\) such that \(\phi=\Phi\circ\pi\), and \(\Phi\) is supported on finitely many bundle pieces, and is smooth along
			the nilmanifold fibres with uniform fibrewise bounds: for every \(K\geq0\),
			\[
			\sup_{j}\operatorname*{ess\,sup}_{b\in B_j}
			\|\Phi(b,\cdot)\|_{C^K(M_j)}<\infty,
			\]
			where the supremum is over the finitely many pieces on which \(\Phi\) is
			non-zero.  When \(\Lambda=\mathbb Z^d\), we simply say that \(\phi\) is a tame
			\(k\)-step nilfunction.
		\end{definition}
		
		\section{Proving \Cref{cor:fixed-directions} assuming \Cref{thm:main-decoupling}}
		\label{sec:corners}
		In this section, we use \Cref{thm:main-decoupling} to imply \Cref{cor:fixed-directions}. The following standard correspondence and recurrence theorems are used to search for the targeted corner configurations.
		\begin{theorem}[Furstenberg's correspondence principle, {\cite[Chapter~7]{Furstenberg1981}}]
			\label{thm:correspondence}
			Fix $\ell \ge 1$. Let $A\subset\Z^\ell$ satisfy $\bd(A)>0$.  Then there exist a standard probability
			space $(X,\cX,\mu)$, commuting invertible measure-preserving transformations
			$T_1,\ldots,T_\ell$, and a set $B\in\cX$ with $\mu(B)=\bd(A)$ such that, for
			every finite $F\subset\Z^\ell$,
			\begin{equation}\label{furs_corr_pro1981}
				\bd\Bigl(\bigcap_{v\in F}(A-v)\Bigr)\geq
				\mu\Bigl(\bigcap_{v=(v_1,\ldots,v_\ell)\in F}
				T_1^{-v_1}\cdots T_\ell^{-v_\ell}B\Bigr).
			\end{equation}
		\end{theorem}
		\begin{theorem}[Furstenberg--Katznelson, {\cite[Theorem~A]{FK}}]
			\label{thm:fk}
			Let $U_1,\ldots,U_m$ be commuting invertible measure-preserving
			transformations on a probability space $(X,\cX,\mu)$.  If $C\in\cX$ and
			$\mu(C)>0$, then
			\[
			\liminf_{N\to\infty}\frac1N\sum_{n=1}^N
			\mu(C\cap U_1^{-n}C\cap\cdots\cap U_m^{-n}C)>0 .
			\]
		\end{theorem}
			Next, we finish the core task of this section.
			\begin{proof}[Proof of \Cref{cor:fixed-directions}]
				Put \(\delta=d^*(A)>0\).  By \Cref{thm:correspondence}, there are commuting
				invertible measure-preserving transformations \(T_1,\ldots,T_\ell\) acting on the standard probability space $(X,\cX,\mu)$ and a set
				\(B\) with \(\mu(B)=\delta\) for which \eqref{furs_corr_pro1981} holds.  For
				\(u=(u_1',\ldots,u_\ell')\in\mathbb Z^\ell\), write
				$
				T^u=T_1^{u_1'}\cdots T_\ell^{u_\ell'},
				$
				and set
				$
				R_1=T^{u_1},R_2=T^{u_2},S=T^{u_3}.
				$
				Let \(q=\E_{\mu}(\one_B\mid\mathcal I(S))\).  Applying
				\Cref{thm:main-decoupling} to $R_1,R_2,S$ and \(f_1=f_2=g=\one_B\) gives
				\begin{align}
					&\frac1N\sum_{n=1}^N
					\mu(B\cap R_1^{-n}B\cap R_2^{-n}B\cap S^{-\Omega(n)}B)
					-
					\frac1N\sum_{n=1}^N
					\int_{X} \one_B\cdot q\,R_1^n\one_B\,R_2^n\one_B\,d\mu
					\longrightarrow0\label{eq:corners-weighted-unweighted}
				\end{align} as $N\to\infty$.
				Moreover, by Jensen's inequality,
				\[
				\int_Bq\,d\mu=\int_{X} q^2\,d\mu\geq \left(\int_{X}qd\mu\right)^{2}=\delta^2.
				\]
				So, \(C:=B\cap\{q\geq\delta/2\}\) has positive measure.  This implies that for every \(n\),
				\[
				\one_B\cdot q\,R_1^n\one_B\,R_2^n\one_B
				\geq \frac\delta2\,\one_C\,R_1^n\one_C\,R_2^n\one_C.
				\]
				Thus \Cref{thm:fk} and \eqref{eq:corners-weighted-unweighted} imply
				\begin{equation}\label{eq:fixed-direction-positive-average}
					\liminf_{N\to\infty}\frac1N\sum_{n=1}^N
					\mu(B\cap R_1^{-n}B\cap R_2^{-n}B\cap S^{-\Omega(n)}B)>0.
				\end{equation}
				Then for every \(n\), the correspondence inequality applied to
				\[
				\{0,nu_1,nu_2,\Omega(n)u_3\}
				\]
				gives
				\[
				d^*\bigl(A\cap(A-nu_1)\cap(A-nu_2)
				\cap(A-\Omega(n)u_3)\bigr)
				\geq
				\mu(B\cap R_1^{-n}B\cap R_2^{-n}B\cap S^{-\Omega(n)}B).
				\]
				Hence
				\eqref{eq:fixed-direction-positive-average} implies that the set on which
				they are positive has positive lower density.
			\end{proof}
			
			\section{Reduction of \Cref{thm:main-decoupling}: characteristic factors}
			\label{sec:orthogonality-austin}
			This section is the first reduction step in the proof of \Cref{thm:main-decoupling}. Its core task is to prove \Cref{prop:prime-dilation-austin-reduction}, which gives the structures of corresponding characteristic factors.
			\subsection{Two preliminary results}
			The first one is a standard orthogonality criterion.
			\begin{lemma}[{Cf. \cite[Theorem 2]{BSZ}}]
				\label{lem:orthogonality}
				Let $(v_n)$ be a bounded sequence in a Hilbert space.  If for a set of primes
				$\cP_0$ of positive lower relative density inside the primes, one has that
				\[
				\lim_{N\to\infty}\left|
				\frac1N\sum_{n=1}^N\langle v_{pn},v_{qn}\rangle
				\right|=0
				\]
				for all distinct $p,q\in\cP_0$, then
				$N^{-1}\sum_{n\leq N}v_n\to0$ in norm.
			\end{lemma}
			The second one describes the characteristic factors of triple commuting ergodic averages.
			\begin{theorem}[Austin, {\cite[Theorem~1.1]{Austin2009II}}]
				\label{thm:austin-pleasant}
				Let \((X,\cX,\mu,(T^v)_{v\in \Z^2})\) be a \(\mathbb Z^2\)-system.
				Then it has an extension
				$
				\pi:(\widetilde X,\widetilde{\cX},\widetilde\mu,(\widetilde T^v)_{v\in \Z^2})
				\to (X,\cX,\mu,(T^v)_{v\in \Z^2})
				$
				and a factor \(\mathcal N_{\widetilde X}\) whose target is an inverse limit of direct
				integrals of two-step \(\mathbb Z^2\)-pronilsystems, such that the
				following holds. For all distinct \(p_1,p_2,p_3\in\mathbb Z^2\setminus\{0\}\) for which
				they are in general position with $0$ \footnote{The vectors \(p_1,p_2,p_3\in\mathbb Z^2\setminus\{0\}\) are in general position with $0$
					if no three of \(0,p_1,p_2,p_3\) lie on a line.}, 
				the ergodic averages
				\[
				\frac1N\sum_{n=1}^N
				\widetilde T^{np_1}f_1\,
				\widetilde T^{np_2}f_2\,
				\widetilde T^{np_3}f_3
				\]
				admit, in the \(i\)-th coordinate, the characteristic factor
				\(\cI(\widetilde T^{p_i})\vee
				\cI(\widetilde T^{p_i-p_j})\vee
				\cI(\widetilde T^{p_i-p_k})\vee\mathcal N_{\widetilde X}\), where
				\(\{i,j,k\}=\{1,2,3\}\).
			\end{theorem}
			\subsection{On the factor $\mathcal N_{\widetilde X}$ in \Cref{thm:austin-pleasant}}
			Austin's theorem produces an inverse limit $\mathcal N_{\widetilde X}$ of direct integrals of two-step
			$\Z^2$-pronilsystems in the sense of \cite[Definition~3.3 and
			Subsection~3.3]{Austin2009II}.
			
			For the inverse limit $Y$ of direct integrals $Y_{n}$ of $2$-step $\Z^2$-pronilsystems, according to \cite[Definition 3.10]{Austin2009II}, we have the following structure:
			$$Y_{0}\longleftarrow Y_{1}\longleftarrow \cdots \longleftarrow Y_{\eta}\longleftarrow Y,$$ where almost every ergodic component of $Y_{n}$ is an ergodic $2$-step $\Z^2$-pronilsystem. 
			
			The above structure motivates the following result, which implies that $\mathcal N_{\widetilde X}\subset Z_{2}(\widetilde{T})$.
			\begin{proposition}
				\label{prop:austin-pronil-domination}
				Let \(\pi:(X,\mathcal X,\mu,(T^v)_{v\in \Z^r})\to
				(Y,\mathcal Y,\nu,(R^v)_{v\in \Z^r})\) be a factor map.  Suppose
				that \(Y\) is an inverse limit of factors \(Y_m\) such that, for every
				\(m\), almost every ergodic component of \(Y_m\) is a \(k\)-step
				pronilsystem.  Then
				$
				\pi^{-1}\mathcal Y\subseteq Z_k(T).
				$
				In particular, the factor $\mathcal N_{\widetilde X}$ in \Cref{thm:austin-pleasant} is contained in
				\(Z_2\) of the corresponding extension.
			\end{proposition}
			Before proving the above, we need a lemma.
			\begin{lemma}
				\label{lem:factor-domination}
				Let \(\pi:(X,\mu,(T^v)_{v\in \Z^r})\to(Y,\nu,(R^v)_{v\in \Z^r})\) be a factor map.  Then, for every \(k\geq0\),
				$
				\pi^{-1}Z_k(R)\subseteq Z_k(T).
				$
			\end{lemma}
			\begin{proof}
				For \(k=0\), the claim is immediate from the fact that for every $R$-invariant $f\in L^{\infty}(\nu)$, $f\circ \pi$ is $T$-invariant.  Assume \(k\geq1\).
				For the ergodic case, it follows from
				\cite[Proposition~A.3.(ii)]{JamneshanShalomTao2026}. (When $r=1$, it is from \cite[Chapter~9, Proposition 11]{HostKraBook}.)
				
				Next, we deal wit the general case. Write the ergodic decompositions as
				\[
				\mu=\int_{\Omega_X}\mu_\omega\,d\lambda(\omega),\ \text{and}\ 
				\nu=\int_{\Omega_Y}\nu_\eta\,d\beta(\eta).
				\]
				For \(\lambda\)-a.e. \(\omega\in \Omega_X\),
				\(\pi_*\mu_\omega\) is ergodic. By the uniqueness of the
				ergodic decomposition, there is, after discarding some null sets, a measurable
				map \(\sigma:\Omega_X\to\Omega_Y\) such that
				$
				\pi_*\mu_\omega=\nu_{\sigma(\omega)}.
				$
				Thus \(\pi\) restricts on almost every ergodic component to a factor map
				$$
				(X,\mu_\omega,(T^v)_{v\in \Z^r})\longrightarrow
				(Y,\nu_{\sigma(\omega)},(R^v)_{v\in \Z^r}).
				$$
				
				Let \(F\in L^\infty(Z_k(R))\).  By
				\cite[Lemma~2.1]{JamneshanMachado2026}, the restriction of \(F\) to
				\((Y,\nu_\eta,(R^v)_{v\in \Z^r})\) is \(Z_k(Y,\nu_\eta,(R^v)_{v\in \Z^r})\)-measurable for
				\(\beta\)-a.e. \(\eta\in \Omega_Y\). Then by the ergodic case, \(F\circ\pi\), restricted to \((X,\mu_\omega,(T^v)_{v\in \Z^r})\), is
				\(Z_k(X,\mu_\omega,(T^v)_{v\in \Z^r})\)-measurable for \(\lambda\)-a.e.
				\(\omega\in \Omega_X\). By \cite[Lemma~2.1]{JamneshanMachado2026} again,
				\(F\circ\pi\) is \(Z_k(T)\)-measurable.  Hence,
				\(\pi^{-1}Z_k(R)\subseteq Z_k(T)\). This proves the lemma.
			\end{proof}
			\begin{proof}[Proof of \Cref{prop:austin-pronil-domination}]
				First, we consider \(Y_m\), and write its ergodic decomposition as
				$
				\displaystyle\nu_m=\int_{\Omega_{m}} \nu_{m,\omega}\,d\beta_m(\omega).
				$
				For \(\beta_m\)-a.e. \(\omega\in \Omega_m\),
				\((Y_m,\mathcal{Y}_{m},\nu_{m,\omega},(R^v)_{v\in \Z^r})\) is a \(k\)-step pronilsystem. By \cite[Theorem~1.1]{JamneshanMachado2026} or \cite[Section~5]{CandelaSzegedy},
				$\mathcal{Y}_{m}$ coincides with $Z_{k}(Y_m,\mathcal{Y}_{m},\nu_{m,\omega},(R^v)_{v\in \Z^r})$ for $(Y_m,\mathcal{Y}_{m},\nu_{m,\omega},(R^v)_{v\in \Z^r})$.  
				
				By \cite[Lemma~2.1]{JamneshanMachado2026},
				for
				\(\beta_m\)-a.e. \(\omega\in \Omega_m\), 
				$$Z_{k}(Y_m,\mathcal{Y}_{m},\nu_{m,\omega},(R^v)_{v\in \Z^r})=Z_{k}(Y_m,\mathcal{Y}_{m},\nu_{m},(R^v)_{v\in \Z^r}).$$
				If \(F\in L^\infty(Y_m,\mathcal{Y}_{m},\nu_m)\), it follows that
				$
				\mathbb E_{\nu_m}(F\mid Z_{k}(Y_m,\mathcal{Y}_{m},\nu_{m},(R^v)_{v\in \Z^r}))=F
				$ in $L^2(\nu_{m,\omega})$
				for $\beta_m$-a.e. \(\omega\in \Omega_m\).  This means that integrating over the ergodic decomposition
				gives the same equality in \(L^2(\nu_m)\).  Hence, for $(Y_m,\mathcal{Y}_{m},\nu_{m},(R^v)_{v\in \Z^r})$,
				$
				\mathcal Y_m=Z_k(Y_m,\mathcal{Y}_{m},\nu_{m},(R^v)_{v\in \Z^r}).
				$
				
				Let \(\rho_m:Y\to Y_m\) be the factor map.  By
				\Cref{lem:factor-domination},
				$
				\rho_m^{-1}\mathcal Y_m\subseteq Z_k(Y,\mathcal Y,\nu,(R^v)_{v\in \Z^r}).
				$
				By the definition of the inverse limit,
				$
				\displaystyle\mathcal Y=\bigvee_m\rho_m^{-1}\mathcal Y_m.
				$
				Hence, \(\mathcal Y\subseteq Z_k(Y,\mathcal Y,\nu,(R^v)_{v\in \Z^r})\). Then  \(\mathcal Y=Z_k(Y,\mathcal Y,\nu,(R^v)_{v\in \Z^r})\) for $(Y,\mathcal Y,\nu,(R^v)_{v\in \Z^r})$. Applying
				\Cref{lem:factor-domination} again finishes the proof.
			\end{proof}
			
			%
			
			\subsection{Two midterm lemmas}
			Now, we give the first lemma.
			\begin{lemma}
				\label{lem:prime-dilation-to-three-direction}
				Let \(T_1,T_2,S\) be three commuting invertible measure-preserving transformations
				of a standard probability space \((X,\mathcal X,\mu)\). Write $	T^{(a,b)}$ for the action $T_1^aT_2^b$. Let $f_1,f_2\in L^\infty(\mu)$, and let $\cP_0$ be a set of primes of positive lower relative density inside the primes such that, for every pair of distinct primes $p,q\in \cP_0$,
				\begin{equation}
					\label{eq:prime-dilation-three-direction}
					\lim_{N\to\infty}\frac1N\sum_{n=1}^N
					\int T^{n(p,-q)}f_1\,
					T^{n(q,-q)}\overline{f_1}\,
					T^{n(0,p-q)}f_2\,\cdot
					\overline{f_2}\,d\mu=0.
				\end{equation}
				Then, for every $g\in L^\infty(\mu)$,
				\[
				\lim_{N\to\infty}\frac1N\sum_{n=1}^N
				T_1^n f_1\,T_2^n f_2\,S^{\Omega(n)}g
				=0
				\]
				in $L^{2}(\mu)$.
			\end{lemma}
			\begin{proof}
				It suffices to consider the case where $|g|\equiv 1$. Put
				$
				v_n=T_1^n f_1\,T_2^n f_2\,S^{\Omega(n)}g .
				$
				Since $\Omega(pn)=\Omega(qn)=\Omega(n)+1$,
				$$
				\langle v_{pn},v_{qn}\rangle
				=
				\int_{X} T_1^{pn}f_1\,\overline{T_1^{qn}f_1}\,
				T_2^{pn}f_2\,\overline{T_2^{qn}f_2}\,d\mu .
				$$
				The $\mu$-invariance of $T_2$ gives
				\[
				\langle v_{pn},v_{qn}\rangle
				=
				\int T^{n(p,-q)}f_1\,
				T^{n(q,-q)}\overline{f_1}\,
				T^{n(0,p-q)}f_2\,
				\overline{f_2}\,d\mu .
				\]
				Hence, \eqref{eq:prime-dilation-three-direction} and \Cref{lem:orthogonality} yield the lemma.
			\end{proof}
			Next, we introduce the second technique lemma, whose proof is standard.
			\begin{lemma}
				\label{lem:S-preserving-lift}
				Fix a \(\mathbb Z^2\)-system \((X,\mathcal X,\mu,(T^v)_{v\in \Z^2})\) and an invertible measure-preserving transformation $S:X\to X$ commuting
				with \(T\). Fix a factor map
				$
				\pi:(Y,\mathcal Y,\nu,({\widetilde{T}}^v)_{v\in \Z^2})\to (X,\mathcal X,\mu,(T^v)_{v\in \Z^2}).
				$
				Then there exist a standard probability
				space \((\widehat X,\widehat{\mathcal X},\widehat\mu)\), a
				measure-preserving \(\mathbb Z^2\)-action \(\widehat T\), and an invertible
				measure-preserving transformation \(\widehat S:\widehat X\to \widehat X\) commuting with \(\widehat T\),
				together with factor maps
				$
				\theta:(\widehat X,\widehat{\mathcal X},\widehat\mu,({\widehat{T}}^v)_{v\in \Z^2})
				\to
				(Y,\mathcal Y,\nu,({\widetilde{T}}^v)_{v\in \Z^2})
				$
				and
				$
				\rho:(\widehat X,\widehat{\mathcal X},\widehat\mu,({\widehat{T}}^v)_{v\in \Z^2})
				\to
				(X,\mathcal X,\mu,(T^v)_{v\in \Z^2}),
				$
				such that \(\pi\circ\theta=\rho\) and $\rho \circ \widehat S=S\circ \rho$. Moreover, \(\mathcal I(\widehat S)=\rho^{-1}\mathcal I(S)\), 
				and for every \(g\in L^2(\mu)\),
				$
				\E_{\widehat{\mu}}(g\circ\rho\mid\mathcal I(\widehat S))
				=
				\E_{\mu}(g\mid\mathcal I(S))\circ\rho .
				$
			\end{lemma}
			\begin{proof}
				Disintegrate \(\nu\) over \(\mu\) with respect to \(\pi\), say that 
				$\displaystyle \nu=\int_X\nu_x\,d\mu(x)$. 
				Define
				$
				\widehat X
				=
				\Bigl\{(x,(y_m)_{m\in\mathbb Z})\in X\times Y^{\mathbb Z}:
				\pi(y_m)=S^m x\ \text{for all }m\in\mathbb Z\Bigr\}.
				$
				On \(X\times Y^{\mathbb Z}\), define
				\[
				\widehat\mu
				=
				\int_X \delta_x\otimes
				\bigotimes_{m\in\mathbb Z}\nu_{S^m x}\,d\mu(x).
				\]
				Since \(S\) preserves \(\mu\), for \(\mu\)-a.e. \(x\in X\),
				\(\nu_{S^m x}(\pi^{-1}\{S^m x\})=1\) for all \(m\in\mathbb Z\).  Hence,
				\(\widehat\mu(\widehat X)=1\).
				
				For each \(v\in\mathbb Z^2\), set
				$$
				\widehat T^v(x,(y_m)_{m\in\mathbb Z})
				=
				\bigl(T^v x,(\widetilde T^v y_m)_{m\in\mathbb Z}\bigr),\ \text{and}\ 
				\widehat S(x,(y_m)_{m\in\mathbb Z})
				=
				\bigl(Sx,(y_{m+1})_{m\in\mathbb Z}\bigr).
				$$
				These maps preserve \(\widehat X\), because
				\(\pi(\widetilde T^v y_m)=T^v\pi(y_m)=T^vS^m x=S^mT^v x\), and
				\(\pi(y_{m+1})=S^{m+1}x=S^m(Sx)\).  Moreover,
				\(\widehat T^v\widehat S=\widehat S\widehat T^v\).
				
				For any \(v\in\mathbb Z^2\),
				\[
				\begin{aligned}
					(\widehat T^v)_*\widehat\mu
					&=
					\int_X
					\delta_{T^v x}\otimes
					\bigotimes_{m\in\mathbb Z}(\widetilde T^v)_*\nu_{S^m x}\,d\mu(x) 
					=
					\int_X
					\delta_{T^v x}\otimes
					\bigotimes_{m\in\mathbb Z}\nu_{T^vS^m x}\,d\mu(x)  \\
					&=
					\int_X
					\delta_{T^v x}\otimes
					\bigotimes_{m\in\mathbb Z}\nu_{S^m(T^v x)}\,d\mu(x) 
					=
					\int_X
					\delta_z\otimes
					\bigotimes_{m\in\mathbb Z}\nu_{S^m z}\,d\mu(z)
					=
					\widehat\mu .
				\end{aligned}
				\]
				Here, the last line uses the change of variables \(z=T^v x\).  
				
				Similarly, write 
				\(\sigma((y_m)_{m\in \Z})=(y_{m+1})_{m\in \Z}\), then
				\[
				\begin{aligned}
					(\widehat S)_*\widehat\mu
					&=
					\int_X
					\delta_{Sx}\otimes
					\sigma_*\Bigl(\bigotimes_{m\in\mathbb Z}\nu_{S^m x}\Bigr)\,d\mu(x)
					=
					\int_X
					\delta_{Sx}\otimes
					\bigotimes_{m\in\mathbb Z}\nu_{S^{m+1}x}\,d\mu(x)  \\
					&=
					\int_X
					\delta_{Sx}\otimes
					\bigotimes_{m\in\mathbb Z}\nu_{S^m(Sx)}\,d\mu(x)  
					=
					\int_X
					\delta_z\otimes
					\bigotimes_{m\in\mathbb Z}\nu_{S^m z}\,d\mu(z)
					=
					\widehat\mu.
				\end{aligned}
				\]
				Thus \(\widehat T\) and \(\widehat S\) are $\widehat \mu$-preserving.
				
				Let $\widehat{\X}=\left(\X\otimes \mathcal{Y}^{\otimes\Z}\right)\Big|_{\widehat{X}}$. 
				Define two measurable maps \(\rho:(\widehat{X},\widehat{\X})\to (X,\X),(x,(y_m)_{m\in \Z})\mapsto x\) and \(\theta:(\widehat{X},\widehat{\X})\to (Y,\mathcal{Y}), (x,(y_m)_{m\in \Z})\mapsto y_0\).
				Since
				$$
				\rho_*\widehat\mu=\int_X\delta_x\,d\mu(x)=\mu,\ \text{and}\ 
				\theta_*\widehat\mu
				=\int_X\nu_x\,d\mu(x)=\nu,
				$$
				\(\rho\) and \(\theta\) are measure-preserving. And
				\(\rho\circ\widehat T^v=T^v\circ\rho\),
				and 
				\(\theta\circ\widehat T^v=\widetilde T^v\circ\theta\).
				Then $\rho$ and $\theta$ are two factor maps. Moreover, \(\rho\circ\widehat S=S\circ\rho\) and \(\pi\circ\theta=\rho\).
				
				Next, we verify the last assertion. Put
				\(\mathcal B=\rho^{-1}\mathcal X\).  Since
				\(\rho\circ\widehat S=S\circ\rho\),
				\(\mathcal B\) is \(\widehat S\)-invariant; hence conditional
				expectation onto \(\mathcal B\) commutes with \(\widehat S\).
				First, every \(\widehat S\)-invariant \(L^2\)-function is
				\(\mathcal B\)-measurable.
				
				For \(m\in\mathbb Z\), let \(p_m:\widehat X\to Y\) be the coordinate projection
				\(p_m(x,(y_j)_{j\in \Z})=y_m\).  For any non-empty finite \(I\subset\mathbb Z\), put
				$\displaystyle\mathcal C_I=\mathcal B\vee\bigvee_{m\in I}p_m^{-1}\mathcal Y.$
				Since
				$
				\widehat{\mathcal X}
				=
				\mathcal B\vee\bigvee_{m\in\mathbb Z}p_m^{-1}\mathcal Y,
				$
				$\{f\in L^{\infty}(\widehat X,\mathcal C_I,\widehat\mu):\varnothing\neq I\subset \Z\ 
				\text{finite}\}$ are dense in \(L^2(\widehat\mu)\).
				
				Let \(F\in L^2(\widehat\mu)\) satisfy \(F\circ\widehat S=F\), and set
				\(F_0=F-\E_{\widehat\mu}(F\mid\mathcal B)\).  Then \(F_0\circ\widehat S=F_0\) and
				\(\E_{\widehat\mu}(F_0\mid\mathcal B)=0\)
				.  Fix a sufficiently small \(\varepsilon>0\).  Choose a  non-empty finite \(I\subset\mathbb Z\) and a bounded
				\(\mathcal C_I\)-measurable function \(C'\)
				such that \(\|F_0-C'\|_{L^{2}(\widehat\mu)}<\varepsilon/2\).  Set
				\(C=C'-\E_{\widehat\mu}(C'\mid\mathcal B)\).  Then \(\E_{\widehat\mu}(C\mid\mathcal B)=0\), and
				\[
				\begin{aligned}
					\|F_0-C\|_{L^{2}(\widehat\mu)}
					=
					\bigl\|F_0-C'+\E_{\widehat\mu}(C'-F_0\mid\mathcal B)\bigr\|_{L^{2}(\widehat\mu)}
					\leq
					2\|F_0-C'\|_{L^{2}(\widehat\mu)}
					<\varepsilon .
				\end{aligned}
				\]
				
				Choose \(k\in\mathbb Z\) such that \(I\cap(I+k)=\varnothing\).  Since
				\(\widehat S^k(x,(y_m)_{m\in \Z})=(S^kx,(y_{m+k})_{m\in \Z})\),
				\(C\circ\widehat S^k\) is \(\mathcal C_{I+k}\)-measurable.  Since \(C\) is
				\(\mathcal C_I\)-measurable, we can write it as \(C(x,(y_m)_{m\in I})\).  Then,
				after identifying \(\mathcal B\)-measurable functions with functions of \(x\),
				\[
				\begin{aligned}
					&\E_{\widehat\mu}\bigl(C\cdot \,\overline{C\circ\widehat S^k}\mid\mathcal B\bigr)(x)  \\
					= &
					\int_{Y^{\Z}}
					C(x,(y_m)_{m\in I})\,
					\overline{C(S^kx,(y_{m+k})_{m\in I})}\,
					d\Bigl(\bigotimes_{m\in\mathbb Z}\nu_{S^m x}\Bigr)((y_m)_{m\in \Z})  \\
					= &
					\left(
					\int_{Y^{I}}
					C(x,(y_m)_{m\in I})\,
					d\Bigl(\bigotimes_{m\in I}\nu_{S^m x}\Bigr)
					\right)
					\left(
					\int_{Y^{I+k}}
					\overline{C(S^kx,(y_m)_{m\in I+k})}\,
					d\Bigl(\bigotimes_{m\in I+k}\nu_{S^m x}\Bigr)
					\right)  \\
					=&
					\E_{\widehat\mu}(C\mid\mathcal B)(x)\,
					\E_{\widehat\mu}(\overline{C\circ\widehat S^k}\mid\mathcal B)(x)
					=0.
				\end{aligned}
				\]
				Hence, \(\langle C,C\circ\widehat S^k\rangle_{L^{2}(\widehat\mu)}=0\).
				
				Since \(F_0\circ\widehat S^k=F_0\), the triangle inequality and
				\(\langle C,C\circ\widehat S^k\rangle_{L^{2}(\widehat\mu)}=0\) give
				\[
				\begin{aligned}
					&\|F_0\|_{L^{2}(\widehat\mu)}^2
					=\bigl|\langle F_0,F_0\circ\widehat S^k\rangle_{L^{2}(\widehat\mu)}\bigr|
					\leq
					\bigl|\langle F_0-C,F_0\circ\widehat S^k\rangle_{L^{2}(\widehat\mu)}\bigr|
					+\bigl|\langle C,(F_0-C)\circ\widehat S^k\rangle_{L^{2}(\widehat\mu)}\bigr|  \\
					\leq &
					\|F_0-C\|_{L^{2}(\widehat\mu)}(\|F_0\|_{L^{2}(\widehat\mu)}+\|C'\|_{L^{2}(\widehat\mu)})
					<
					\varepsilon\|F_0\|_{L^{2}(\widehat\mu)}+\varepsilon(\|F_0\|_{L^{2}(\widehat\mu)}+\varepsilon/2) .
				\end{aligned}
				\]
				This means that every \(\widehat S\)-invariant \(L^2\)-function is \(\rho^{-1}\mathcal X\)-measurable. That is, $\mathcal I(\widehat S)\subseteq \rho^{-1}\mathcal X$. Let
				\(F\in L^2(\widehat\mu)\) be \(\widehat S\)-invariant. Then \(F=H\circ\rho\) for some \(H\in L^2(\mu)\).
				The \(\widehat S\)-invariance of
				\(F\) gives \((H\circ S)\circ\rho=H\circ\rho\).  Since
				\(\rho_*\widehat\mu=\mu\), it follows that \(H\circ S=H\) in \(L^2(\mu)\).
				Therefore,
				\(\mathcal I(\widehat S)\subseteq\rho^{-1}\mathcal I(S)\). Since \(\rho\circ\widehat S=S\circ\rho\), one has
				\(\rho^{-1}\mathcal I(S)\subseteq\mathcal I(\widehat S)\). To sum up, \(\rho^{-1}\mathcal I(S)=\mathcal I(\widehat S)\).
				Consequently, for every \(g\in L^2(\mu)\),
				\(\E_{\widehat\mu}(g\circ\rho\mid\mathcal I(\widehat S))
				=\E_{\mu}(g\mid\mathcal I(S))\circ\rho\). The proof is complete.
			\end{proof}
			\subsection{First reduction}
			Let \(\mathbb P\) denote the set of primes.  Now, based on \Cref{thm:austin-pleasant} and \Cref{lem:prime-dilation-to-three-direction}, we give the following notations: For distinct primes \(p,q\), set
			\[
			\begin{aligned}
				W_1^{p,q}
				&=\{(p,-q),(p-q,0),(p,-p),(q,-q),(q-p,0),(q,-p)\},\\
				W_2^{p,q}
				&=\{(0,p-q),(-p,p),(-q,p)\}.
			\end{aligned}
			\]
			For a \(\mathbb Z^2\)-system \((X,\cX,\mu,(T^v)_{v\in \Z^2})\) and a
			\(T\)-invariant factor \(\mathcal N_X\subseteq Z_2(T)\), set
			\begin{equation}\label{eq:char-factors}
				\mathcal Z_{i,\mathcal N_X}
				=
				\mathcal N_X
				\vee
				\bigvee_{\substack{p,q\in\mathbb P\\p\neq q}}
				\bigvee_{w\in W_i^{p,q}}
				\cI(T^w),
				i=1,2.
			\end{equation}
			\begin{proposition}
				\label{prop:prime-dilation-austin-reduction}
				To prove \(\Cref{thm:main-decoupling}\), it suffices to prove that, for
				any \(\mathbb Z^2\)-system \((X,\cX,\mu,(T^v)_{v\in \Z^2})\), any
				invertible measure-preserving transformation \(S:X\to X\) commuting with \(T\),
				any \(T\)-invariant factor \(\mathcal N_X\subseteq Z_2(T)\), one has that for all
				\(f_i\in L^\infty(X,\mathcal Z_{i,\mathcal N_X},\mu)\), \(i=1,2\), and all
				\(g\in L^\infty(\mu)\),
				\[
				\frac1N\sum_{n=1}^N
				T^{n(1,0)}f_1\,T^{n(0,1)}f_2\,S^{\Omega(n)}g
				-
				\left(
				\frac1N\sum_{n=1}^N
				T^{n(1,0)}f_1\,T^{n(0,1)}f_2
				\right)
				\E_{\mu}(g\mid\cI(S))
				\longrightarrow0
				\]
				in \(L^2(\mu)\) as $N\to\infty$.
			\end{proposition}
			
			\begin{proof}
				Assume the structured statement in the above proposition. Next, we prove that  \(\Cref{thm:main-decoupling}\) holds.
				
				Let \(T_1,T_2,S\) be as in \Cref{thm:main-decoupling}, and define the
				\(\mathbb Z^2\)-action \(T\) by \(T^{(a,b)}=T_1^aT_2^b\).  Then \(S\)
				commutes with \(T\).  Apply \Cref{thm:austin-pleasant} to
				\((X,\cX,\mu,(T^v)_{v\in \Z^2})\), obtaining an extension
				\[
				\pi:(Y,\cY,\nu,({\widetilde{ T}}^v)_{v\in \Z^2})\to (X,\cX,\mu,T)
				\]
				and the Austin's factor \(\mathcal N_Y\).  By
				\Cref{prop:austin-pronil-domination},
				$
				\mathcal N_Y\subseteq Z_2(\widetilde T).
				$
				Let
				\(\mathcal Z_{i,Y}\), \(i=1,2\), be the factors defined by
				\eqref{eq:char-factors} for \((Y,\cY,\nu,({\widetilde{ T}}^v)_{v\in \Z^2})\) and
				\(\mathcal N_Y\).
				
				Apply \Cref{lem:S-preserving-lift} to \(\pi\) and \(S\), obtaining the standard probability space 
				\((\widehat X,\widehat{\cX},\widehat\mu)\), a \(\mathbb Z^2\)-action
				\(\widehat T\) on $\widehat X$, an invertible measure-preserving transformation \(\widehat S:\widehat X\to \widehat X\) commuting with
				\(\widehat T\), and two factor maps
				\[
				\theta:(\widehat X,\widehat{\cX},\widehat\mu,({\widehat{ T}}^v)_{v\in \Z^2})
				\to (Y,\cY,\nu,({\widetilde{ T}}^v)_{v\in \Z^2}),
				\rho:(\widehat X,\widehat{\cX},\widehat\mu,({\widehat{ T}}^v)_{v\in \Z^2})
				\to (X,\cX,\mu,(T^v)_{v\in \Z^2})
				\]
				such that \(\pi\circ\theta=\rho\), $\rho \circ \widehat S=S\circ \rho$ and
				\(\mathcal I(\widehat S)=\rho^{-1}\mathcal I(S)\).
				
				Set \(\mathcal N_{\widehat X}=\theta^{-1}\mathcal N_Y\).  Since
				\(\mathcal N_Y\subseteq Z_2(\widetilde T)\),
				\Cref{lem:factor-domination} gives
				$
				\mathcal N_{\widehat X}\subseteq Z_2(\widehat T).
				$
				Let
				\(\mathcal Z_{i,\widehat X}\), \(i=1,2\), be the factors defined by
				\eqref{eq:char-factors} for
				\((\widehat X,\widehat{\cX},\widehat\mu,({\widehat{ T}}^v)_{v\in \Z^2})\) and
				\(\mathcal N_{\widehat X}\).  Since \(\theta\) intertwines \(\widehat T\) and
				\(\widetilde T\), one has
				$
				\theta^{-1}\mathcal Z_{i,Y}\subseteq \mathcal Z_{i,\widehat X},
				i=1,2.
				$
				
				Take \(f_1,f_2,g\in L^\infty(\mu)\). By \Cref{lem:S-preserving-lift}, one has
				\(\E_{\mu}(g\mid\mathcal I(S))\circ\rho
				=\E_{\widehat\mu}(g\circ\rho\mid\mathcal I(\widehat S))\).  Put
				$
				G_0=g\circ\rho-\E_{\widehat\mu}(g\circ\rho\mid\mathcal I(\widehat S)).
				$
				Note that $\rho_{*}\widehat{\mu}=\mu$ and \(G_0=(g-\E_{\mu}(g\mid\mathcal I(S)))\circ\rho\).
				Hence, the proof of \Cref{thm:main-decoupling} is reduced to proving
				\begin{equation}\label{eq:lifted-centered-average}
					\lim_{N\to\infty}
					\norm{
						\frac1N\sum_{n=1}^N
						\widehat T^{n(1,0)}(f_1\circ\rho)\,
						\widehat T^{n(0,1)}(f_2\circ\rho)\,
						\widehat S^{\Omega(n)}G_0
					}_{L^2(\widehat\mu)}
					=0.
				\end{equation}
				
				Write
				$
				f_i\circ\pi
				=
				\E_{\nu}(f_i\circ\pi\mid\mathcal Z_{i,Y})+u_i$, where 
				$
				\E_{\nu}(u_i\mid\mathcal Z_{i,Y})=0.
				$
				Then
				$
				f_i\circ\rho=\phi_i+u_i\circ\theta$, where
				$
				\phi_i=\E_{\nu}(f_i\circ\pi\mid\mathcal Z_{i,Y})\circ\theta .
				$
				Since \(\theta^{-1}\mathcal Z_{i,Y}\subseteq\mathcal Z_{i,\widehat X}\),
				we have \(\phi_i\in L^\infty(\mathcal Z_{i,\widehat X})\).  Also
				\(\E_{\widehat \mu}(G_0\mid\mathcal I(\widehat S))=0\).  Hence, by the starting  hypothesis of the proof, 
				\begin{equation}\label{eq:structured-part-vanishes}
					\lim_{N\to\infty}\frac1N\sum_{n=1}^N
					\widehat T^{n(1,0)}\phi_1\,
					\widehat T^{n(0,1)}\phi_2\,
					\widehat S^{\Omega(n)}G_0=0
					\quad\text{in }L^2(\widehat\mu).
				\end{equation}
				
				For extra three error terms, we need the following claim to eliminate them.
				
				\medskip
				\noindent\emph{Claim.}
				Let \(h_1,h_2\in L^\infty(\nu)\).  If
				\(\E_{\nu}(h_1\mid\mathcal Z_{1,Y})=0\) or
				\(\E_{\nu}(h_2\mid\mathcal Z_{2,Y})=0\), then, for every
				\(G\in L^\infty(\widehat\mu)\),
				\[
				\lim_{N\to\infty}\frac1N\sum_{n=1}^N
				\widehat T^{n(1,0)}(h_1\circ\theta)\,
				\widehat T^{n(0,1)}(h_2\circ\theta)\,
				\widehat S^{\Omega(n)}G
				=0
				\quad\text{in }L^2(\widehat\mu).
				\]
				\medskip
				
				\begin{proof}[Proof of the claim]
					By \Cref{lem:prime-dilation-to-three-direction}, it suffices to prove that, for
					distinct primes \(p,q\),
					\[
					\lim_{N\to\infty}\frac1N\sum_{n=1}^N
					\int_{\widehat X}
					\widehat T^{n(p,-q)}(h_1\circ\theta)\,
					\widehat T^{n(q,-q)}(\overline{h_1}\circ\theta)\,
					\widehat T^{n(0,p-q)}(h_2\circ\theta)\,\cdot
					(\overline{h_2}\circ\theta)\,d\widehat\mu
					=0.
					\]
					Since \(\theta_*\widehat\mu=\nu\), we need to prove that 
					\[
					\lim_{N\to\infty}\frac1N\sum_{n=1}^N
					\int_Y
					\widetilde T^{n(p,-q)}h_1\,
					\widetilde T^{n(q,-q)}\overline{h_1}\,
					\widetilde T^{n(0,p-q)}h_2\,\cdot
					\overline{h_2}\,d\nu=0.
					\]
					Put \(v_1=(p,-q)\), \(v_2=(q,-q)\), and \(v_3=(0,p-q)\).  The points
					\(v_1,v_2,v_3\) are in general position with $0$. By \Cref{thm:austin-pleasant} and \eqref{eq:char-factors}, if
					\(\E_{\nu}(h_i\mid\mathcal Z_{i,Y})=0\), then 
					$$\lim_{N\to\infty}\frac1N\sum_{n=1}^N
					\widetilde T^{n(p,-q)}h_1\,
					\widetilde T^{n(q,-q)}\overline{h_1}\,
					\widetilde T^{n(0,p-q)}h_2=0$$ in $L^{2}(\nu)$. This proves the claim.
				\end{proof}
				Applying the claim with \(G=G_0\), first to
				\((h_1,h_2)=(\E_{\nu}(f_1\circ\pi\mid\mathcal Z_{1,Y}),u_2)\), second to
				\((h_1,h_2)=(u_1,\E_{\nu}(f_2\circ\pi\mid\mathcal Z_{2,Y}))\), and third to $(u_1,u_2)$ shows that the
				three triple ergodic averages containing \(u_1\) or \(u_2\) vanish in
				\(L^2(\widehat\mu)\).  Combining this with
				\eqref{eq:structured-part-vanishes} and
				\(f_i\circ\rho=\phi_i+u_i\circ\theta\) gives
				\eqref{eq:lifted-centered-average}. The proof is complete.
			\end{proof}
			
			\section{Weighted multiple one-action ergodic averages}
			\label{sec:nil-one-action}
			
			This section collects some useful lemmas and isolates the one-action engine after the characteristic factor reduction, which are crucial ingredients in the proofs of \Cref{prop:cyclic-tower-one-action-decoupling} and \Cref{prop:structured-decoupling}.
			%
			\subsection{Approximation by tame nilfunctions}
			\label{subsec:approximation-nil-factors}
			\begin{lemma}
				\label{lem:smooth-bundle-approx}
				Let \(k\geq1\), \((X,\mathcal X,\mu,(T^v)_{v\in \Z^d})\) be a
				\(\mathbb Z^d\)-system,
				\(\Lambda\leq\mathbb Z^d\),
				and
				\(\mathcal N\subseteq Z_k(T|_\Lambda)\) be a
				\(T|_\Lambda\)-invariant subfactor, and
				\(\varnothing\neq \mathcal F\subseteq L^\infty(X,\mathcal N,\mu)\) be finite.  For every
				\(\varepsilon>0\), there are a factor $Y$, which is 
				a bundle of \(k\)-step \(\Lambda\)-nilsystems, associated with a factor map
				$
				\pi:(X,\mathcal X,\mu,(T^v)_{v\in \Lambda})\to(Y,\mathcal Y,\nu,(R^v)_{v\in \Lambda}),
				$
				and \(\Phi_f\in L^\infty(\nu)\), \(f\in\mathcal F\), such that
				\(\Phi_f\circ\pi\) is a tame \(k\)-step nilfunction for
				\(T|_\Lambda\), 
				and
				$
				\|f-\Phi_f\circ\pi\|_{L^2(\mu)}<\varepsilon,\ \text{and}\ 
				\|\Phi_f\|_{L^{\infty}(\nu)}\leq\|f\|_{L^{\infty}(\mu)}+\varepsilon .
				$
			\end{lemma}
			
			\begin{proof}
				Set
				\(\mathcal Z=Z_k(T|_\Lambda)\).  By
				\Cref{thm:jm-structure}, there is an inverse system of \(k\)-step bundle nilfactors 
				$(Y_1,\mathcal Y_1,\nu_1,(T^v)_{v\in \Lambda})\longleftarrow (Y_2,\mathcal Y_2,\nu_2,(T^v)_{v\in \Lambda})\longleftarrow \cdots \longleftarrow (X,\mathcal X,\mu,(T^v)_{v\in \Lambda})$
				such that 
				\(\mathcal Z\) is the inverse limit of the inverse system. Then
				$
				\mathcal Z=\bigvee_{m}\mathcal Y_m.
				$
				Since
				\(\mathcal F\subseteq L^\infty(X,\mathcal N,\mu)\) and
				\(\mathcal N\subseteq\mathcal Z\), the Martingale Theorem gives
				$
				\E_{\mu}(f\mid\mathcal Y_m)\to f
				$
				in $L^2(\mu)$ for every \(f\in\mathcal F\).  As \(\mathcal F\) is finite, choose \(m\) such
				that for all  \(f\in\mathcal F\),
				$
				\|f-\E(f\mid\mathcal Y_m)\|_{L^2(\mu)}<\varepsilon/2$.
				
				Let
				$
				\pi_m:(X,\mathcal X,\mu,(T^v)_{v\in \Lambda})\to(Y_m,\mathcal Y_m,\nu_m,(T^v)_{v\in \Lambda})
				$
				be the corresponding factor map.  Write
				$
				\E_{\mu}(f\mid\mathcal Y_m)=H_f\circ\pi_m$ for some
				$H_f\in L^\infty(\nu_m)$.
				Then
				$
				\|H_f\|_{L^{\infty}(\nu_m)}\leq \|f\|_{L^{\infty}(\mu)} .
				$
				
				By \Cref{def:jm-bundle-nilsystems}, $Y_m$ has the following bundle
				representation over the base space $(\Omega,\mathcal{B},\beta)$:
				$$
				Y_m=\bigsqcup_{j\ge0}B_j\times M_j,M_j=G_j/\Gamma_j,\ \text{with}\ 
				\nu_m=\sum_{j\ge0}\beta|_{B_j}\otimes m_{M_j}.
				$$
				Fix a smooth Riemannian metric $g_j$ on each \(M_j\), and put
				$
				A=1+\max_{f\in\mathcal F}\|H_f\|_{L^{\infty}(\nu_m)} .
				$
				Choose a finite set \(J\subset\{0,1,2,\ldots\}\) such that
				$
				A^2\sum_{j\notin J}\beta(B_j)<\varepsilon^2/16 .
				$ This is possible by the setting that \(\{B_j\}_{j\ge0}\) a measurable partition of \(\Omega\).
				
				For each \(f\in\mathcal F\), choose a smooth Lipschitz map
				\(\tau_f:\mathbb C\to\mathbb C\), identifying \(\mathbb C\) with
				\(\mathbb R^2\), such that
				$
				\tau_f(z)=z\ \text{if } |z|\leq\|H_f\|_{L^{\infty}(\nu_m)};
				|\tau_f(z)|\leq \|H_f\|_{L^{\infty}(\nu_m)}+\varepsilon\ \text{for all }z\in\mathbb C .
				$
				Let \(L_f\) be the correspoding Lipschitz constant of \(\tau_f\).  For each \(j\in J\),
				the linear span of functions shaped like
				$
				(b,y)\mapsto a(b)\psi(y)
				(a\in L^\infty(B_j),\psi\in C^\infty(M_j))
				$
				is dense in \(L^2(B_j\times M_j)\).  Hence, for every \(f\in\mathcal F\) and
				\(j\in J\), we may choose
				$$
				\Psi_{f,j}(b,y)=
				\sum_{\ell=1}^{r_{f,j}} a_{f,j,\ell}(b)\psi_{f,j,\ell}(y),
				a_{f,j,\ell}\in L^\infty(B_j),
				\psi_{f,j,\ell}\in C^\infty(M_j),
				$$
				so that
				\[
				\sum_{j\in J}\|\Psi_{f,j}-H_f\|_{L^2(B_j\times M_j)}^2
				<\frac{\varepsilon^2}{16(1+L_f)^2}.
				\]
				Define, for \(b\in B_j\),
				\[
				\Phi_f(b,y)=
				\begin{cases}
					\tau_f(\Psi_{f,j}(b,y)), & j\in J,\\
					0, & j\notin J .
				\end{cases}
				\]
				Then \(\Phi_f\in L^\infty(\nu_m)\) and
				$
				\|\Phi_f\|_{L^\infty(\nu_m)}\leq \|H_f\|_{L^\infty(\nu_m)}+\varepsilon
				\leq \|f\|_{L^\infty(\mu)}+\varepsilon .
				$
				By the definition of $\tau_f$, for every \(j\), the fibre function \(\Phi_f(b,\cdot)\) is smooth on \(M_j\) for \(\beta\)-a.e. \(b\in B_j\). By the definition of $\Psi_{f,j}$ and the fact that  \(J\) and \(\mathcal F\) are
				finite, 
				for every \(K'\geq0\),
				\[
				\sup_{f\in\mathcal F}\sup_{j\ge0}
				\operatorname*{ess\,sup}_{b\in B_j}
				\|\Phi_f(b,\cdot)\|_{C^{K'}(M_j)}<\infty .
				\]
				Moreover, \(H_f=\tau_f(H_f)\) for $\nu_m$-a.e. $z\in Y_m$; hence,
				$
				\sum_{j\in J}\|\Phi_f-H_f\|_{L^2(B_j\times M_j)}^2
				<\varepsilon^2/16 .
				$
				On the remaining bundle pieces, \(\Phi_f=0\). So, the choice of \(J\) gives
				\[
				\sum_{j\notin J}\|\Phi_f-H_f\|_{L^2(B_j\times M_j)}^2
				\leq \|H_f\|_{L^\infty(\nu_m)}^{2}\sum_{j\notin J}\beta(B_j)<A^2\sum_{j\notin J}\beta(B_j)
				<\varepsilon^2/16 .
				\]
				Thus
				$
				\|H_f-\Phi_f\|_{L^2(\nu_m)}<\varepsilon/2\ \text{and}
				$
				\[
				\begin{aligned}
					\|f-\Phi_f\circ\pi_m\|_{L^2(\mu)}
					&\leq
					\|f-H_f\circ\pi_m\|_{L^2(\mu)}
					+
					\|(H_f-\Phi_f)\circ\pi_m\|_{L^2(\mu)}  \\
					&=
					\|f-\E_{\mu}(f\mid\mathcal Y_m)\|_{L^2(\mu)}
					+
					\|H_f-\Phi_f\|_{L^2(\nu_m)}
					<\varepsilon .
				\end{aligned}
				\]
				The proof is complete.
			\end{proof}
			The following lemma give a characterization of the product of finitely many tame nilfunctions.
			\begin{lemma}
				\label{lem:bundle-weight-ezk-stable}
				Let \((X,\mathcal X,\mu,(T^v)_{v\in \Z^r})\) be a \(\mathbb Z^r\)-system,
				let \(v,u_1,\ldots,u_Q\in\mathbb Z^r\), and put \(U=T^v\).  Let
				\(\phi_1,\ldots,\phi_Q\) be tame nilfunctions. 
				(They may come from different bundle nilfactors).  
				Then
				\[
				c(x,n):=\prod_{q=1}^{Q}\phi_q(T^{nu_q}x),
				(x,n)\in X\times\mathbb Z,
				\]
				is finitely Eisner--Zorin-Kranich stable with respect to \(U\) in the sense of \Cref{def:finite-ezk-stability}.
			\end{lemma}
			
			\begin{proof}
				By the definitions of bundle nilfactor and tame nilfunction (see \Cref{def:jm-bundle-nilsystems} and \Cref{def:tame-finite-stage-nilfunction}), for each \(q\), we have the following structures:
				\[
				\pi_q:X\longrightarrow
				Y_q=\bigsqcup_{\ell\geq0}B_{q,\ell}\times M_{q,\ell}, M_{q,\ell}=N_{q,\ell}/\Delta_{q,\ell},
				\]
				and \(\Phi_q\), which is supported on  finitely many bundle pieces and is smooth along the nilmanifold fibres with uniform fibrewise bounds, such that
				\(\phi_q=\Phi_q\circ\pi_q\).  Let \(E_q\) be the finite set of bundle
				pieces on which \(\Phi_q\) is supported.  On
				\(B_{q,\ell}\times M_{q,\ell}\), write the action as
				\[
				H_q^t(b,y)=(b,\rho_{q,\ell,b}(t)y),
				\rho_{q,\ell,b}\in \text{Hom}(\mathbb Z^r,N_{q,\ell}).
				\]
				
				Fix \(R\geq1\), \(\boldsymbol\epsilon\in\{0,1\}^R\), and triples
				\((r_i,t_i,a_i)\in\mathbb Z^3\), \(1\leq i\leq R\).  For
				\(\mathbf h=(h_1,\ldots,h_R)\), 
				\begin{equation}\label{eq:bundle-derived-separate-factors}
					\begin{aligned}
						c_{\mathbf h}^{\boldsymbol\epsilon}(x,n)
						&:=\prod_{i=1}^{R}\prod_{q=1}^{Q}
						\mathcal C^{\epsilon_i}
						\Phi_q\Bigl(
						H_q^{(r_i n+t_i)v+(a_i n+h_i)u_q}\pi_q(x)
						\Bigr).
					\end{aligned}
				\end{equation}
				
				If \(\pi_q(x)\) lies outside the finite support pieces for some \(q\), then the
				sequence in \eqref{eq:bundle-derived-separate-factors} is identically zero
				and can be represented by a trivival Eisner--Zorin-Kranich admissible filtered nilmanifold.  It remains to consider
				points for which
				$
				\pi_q(x)=(b_q,y_q)\in B_{q,\ell_q}\times M_{q,\ell_q},
				\ell_q\in E_q,
				$
				for every \(q\).
				
				For a fixed tuple \(\boldsymbol\ell=(\ell_1,\ldots,\ell_Q)\in {\bf E}:=E_1\times \cdots \times E_{Q}\), we assign a nilmanifold
				$
				M_{\boldsymbol\ell}
				:=\prod_{q=1}^{Q}M_{q,\ell_q}^{R}.
				$
				Its coordinates are indexed by \((i,q)\).  Put
				$
				\xi_{i,q}=r_i v+a_i u_q,
				\eta_{i,q}(\mathbf h)=t_i v+h_i u_q.
				$
				Since each \(\rho_{q,\ell_q,b_q}\) is a homomorphism, the product of the
				fibre points occurring in \eqref{eq:bundle-derived-separate-factors} is the
				single linear orbit in $M_{\boldsymbol\ell}$,
				where the \((i,q)\)-coordinates of the translating element and the
				initial point are respectively
				$
				\rho_{q,\ell_q,b_q}(\xi_{i,q})\ \text{and}\ 
				\rho_{q,\ell_q,b_q}(\eta_{i,q}(\mathbf h))y_q.
				$
				On \(M_{\boldsymbol\ell}\), for each ${\bf b}\in {\bf B}_{\ell}:=B_{1,\ell_1}\times \cdots\times B_{Q,\ell_Q}$, we put
				\[
				F_{\ell,\mathbf b}\bigl((z_{i,q})_{i,q}\bigr)
				=\prod_{i=1}^{R}\prod_{q=1}^{Q}
				\mathcal C^{\epsilon_i}\Phi_q(b_q,z_{i,q}).
				\]
				
				Apply \Cref{lem:malcev-orbit-replacement} to each \(M_{\boldsymbol\ell}\).
				Then by \eqref{eq:malcev-linear-orbit-lift}, the lemma lifts the targeted 
				linear orbit in $M_{\boldsymbol\ell}$ to some polynomial orbit on its some Eisner--Zorin-Kranich admissible representation $(G_{\ell}/\Gamma_{\ell},G_{\ell,\bullet},\mathcal M_{\ell})$ with the induced map $p_{\ell}:G_{\ell}/\Gamma_{\ell}\to M_{\boldsymbol\ell}$.
				Collecting these Eisner--Zorin-Kranich admissible representations gives a finite family $\mathscr N:=\{(G_{\ell}/\Gamma_{\ell},G_{\ell,\bullet},\mathcal M_{\ell}):\ell\in {\bf E}\}$.
				Now, based on $\mathscr N$ and the family ${\mathscr F}:=\{F_{\ell,\mathbf b}:\ell\in {\bf E},{\bf b}\in {\bf B}_{\ell}\}$, we choose three parameters $s,K,B$. Take the maximum of filtration degrees of all elements of $\mathscr N$ as $s$. Based on $\mathscr N$ and the requirement \eqref{eq:finite-ezk-order-condition}, choose \(K\). Put 
				$$B=1+	\sup_{\ell\in {\bf E}}\operatorname*{ess\,sup}_{{\bf b}\in {\bf B}_{\ell}}
				\|F_{\ell,\mathbf b}\circ p_{\ell}\|_{C^K(G_{\ell}/\Gamma_{\ell})}.$$
				By the choice of each $\Phi_q$ and \eqref{eq:malcev-Ck-pullback}, $0<B<\infty$. Clearly, three parameters $s,K,B$ are independent of \(\mathbf h\).
				
				Now, add the trivival Eisner--Zorin-Kranich admissible filtered nilmanifold to $\mathscr N$ to get the new finite family ${\mathscr N}'$. Then by the above constructions of $s,K,B,{\mathscr N}'$, for every ${\bf h}$, $ c_{\mathbf h}^{\boldsymbol\epsilon}$ has a finite Eisner--Zorin-Kranich representation with parameters \((s,K,B)\) and in this representation, all Eisner--Zorin-Kranich admissible filtered nilmanifolds are from ${\mathscr N}'$.
				This finishes the proof.
				%
			\end{proof}
			\subsection{Two classes of weighted ergodic averages}
			\label{subsec:weighted-one-dimensional-estimates}
			In this subsection, we give the related characteristic factors for two classes of weighted ergodic averages. The first one is on non--conventional ergodic averages weighted by Eisner--Zorin-Kranich regular items. 
			\begin{proposition}
				\label{prop:multislope-nilsequence-estimate}
				Let \((X,\mathcal X,\mu,U)\) be a $\Z$-system,
				\(A\subset\mathbb Z\setminus\{0\}\) be finite and non-empty, and
				\(c:X\times\mathbb Z\to\mathbb C\) be a bounded measurable weight, which is
				Eisner--Zorin-Kranich regular with respect to \(U\) in the sense of
				\Cref{def:finite-ezk-stability}.
				Let \(H_a\in L^\infty(\mu)\), \(a\in A\), with \(\|H_a\|_{L^\infty(\mu)}\le1\).  If
				\(\E_{\mu}(H_{a_0}\mid Z_\infty(U))=0\) for some \(a_0\in A\), then
				\[
				\frac1N\sum_{n=1}^N c(x,n)\prod_{a\in A}U^{an}H_a
				\longrightarrow0
				\]
				in \(L^2(\mu)\) as $N\to\infty$.
			\end{proposition}
			\begin{proof}
				We prove by induction on \(|A|\) that the following hold:
				
				({\bf S}) Suppose that $c$ is bounded by 1. Then there are \(r\ge1\),
				\(C>0\), and \(\kappa>0\) such that
				\[
				\limsup_{N\to\infty}
				\left\|
				\frac1N\sum_{n=1}^N c(x,n)\prod_{a\in A}U^{an}H_a
				\right\|_{L^{2}(\mu)}
				\le
				C\min_{a\in A}\|H_{a}\|_{U^r(U)}^\kappa .
				\]
				The constants depend only on \(A\) and the Eisner--Zorin-Kranich regularity data of $c$
				needed in the induction. 
				
				Clearly, ({\bf S}) can finish the proof by the definition of $Z_\infty(U)$.
				For \(|A|=1\), we apply \Cref{def:finite-ezk-stability} with
				\(R=1\), \(\epsilon_1=0\), \((r_1,t_1,a_1)=(0,0,1)\), and \(h_1=0\), and
				then use \Cref{lem:one-slope}. Suppose that ({\bf S}) holds for all $k$-element subsets of $\Z\backslash \{0\}$. Assume that $|A|=k+1$. Choose \(a_*\in A\), and put
				\[
				u_n(x)=c(x,n)\prod_{a\in A}U^{an}H_a(x).
				\]
				By the van der Corput Lemma, we have that
				\[
				\limsup_{N\to\infty}\left\|\frac1N\sum_{n=1}^Nu_n\right\|_{L^2(\mu)}^2
				\le 
				C_1\limsup_{H\to\infty}H^{-1}\sum_{h=1}^H
				\limsup_{N\to\infty}\left|
				\frac1N\sum_{n=1}^N\langle u_{n+h},u_n\rangle
				\right|,
				\] where the constant $C_1$ depends only on the bound of $c$.
				Set \(G_{a,h}=U^{ah}H_a\cdot\overline{H_a}\).  The change of variables
				\(x=U^{-a_*n}y\) gives
				\[
				\frac1N\sum_{n=1}^N\langle u_{n+h},u_n\rangle
				=
				\int G_{a_*,h}B_{N,h}\,d\mu,
				\]
				where
				\[
				B_{N,h}(y)=
				\frac1N\sum_{n=1}^N C_{h}(y,n)
				\prod_{a\in A\setminus\{a_*\}}U^{(a-a_*)n}G_{a,h}(y),
				\]
				\[
				C_{h}(y,n)=c(U^{-a_*n}y,n+h)\overline{c(U^{-a_*n}y,n)}.
				\]
				By \Cref{def:finite-ezk-stability}, the weights
				\(C_h\) are Eisner--Zorin-Kranich regular with representation data independent of \(h\).  Applying the induction hypothesis to
				$
				A'=\{a-a_*:a\in A,\ a\ne a_*\},
				$
				we get
				\[
				\limsup_{N\to\infty}\|B_{N,h}\|_{L^2(\mu)}
				\le
				C'\min_{a\in A\setminus\{a_*\}}\|G_{a,h}\|_{U^{r'}(U)}^{\kappa'} .
				\]
				By \cite[Chapter 8, Lemma 12]{HostKraBook}, we may assume that $2^{r'}>\kappa'$.
				Hence,
				\[
				\limsup_{N\to\infty}\left\|\frac1N\sum_{n=1}^Nu_n\right\|_{L^2(\mu)}^2
				\le
				C_1C'\min_{a\in A\setminus\{a_*\}}\limsup_{H\to\infty}H^{-1}\sum_{h=1}^H 
				\|U^{ah}H_{a}\overline{H_{a}}\|_{U^{r'}(U)}^{\kappa'}.
				\]
				By the H\"older inequality and \Cref{lem:one-dimensional-power-comparison}, for any $a\in A\setminus\{a_*\}$,
				\[
				\limsup_{H\to\infty}
				\frac1H\sum_{h=1}^H
				\|U^{ah}H_{a}\overline{H_{a}}\|_{U^{r'}(U)}^{\kappa'}\le 
				C_{2}\|H_{a}\|_{U^{r'+1}(U)}^{2\kappa'},
				\]
				where $C_2=C_2(a,\kappa',r')$.  
				To sum up, 
				$$	\limsup_{N\to\infty}\left\|\frac1N\sum_{n=1}^Nu_n\right\|_{L^2(\mu)}^2
				\le C_1C'\max_{a\in A\setminus\{a_*\}}C_2(a,\kappa',r')\min_{a\in A\setminus\{a_*\}}\|H_{a}\|_{U^{r'+1}(U)}^{2\kappa'}.$$ 
				
				Choose $a'\in A$ with $a'\neq a_{*}$. By following the above arguments, we have that 
				$$	\limsup_{N\to\infty}\left\|\frac1N\sum_{n=1}^Nu_n\right\|_{L^2(\mu)}^2
				\le C_1C''\max_{a\in A\setminus\{a'\}}C_2(a,\kappa'',r'')\min_{a\in A\setminus\{a'\}}\|H_{a}\|_{U^{r''+1}(U)}^{2\kappa''},$$ where parameters $C'',r'',\kappa''$ are from the application of  the induction hypothesis.
				By the above two estimates, \cite[Chapter 8, Lemma 12]{HostKraBook}, and the fact that for each $a\in A,d\in \N$, $\|H_{a}\|_{U^{d}(U)}\le 1$, ({\bf S}) holds for $A$.
				This finishes the proof.
			\end{proof}
			
			The second one is on ergodic averages along $\Omega(n)$ weighted by some analogue of nilsequences. Before stating it, we introduce an ergodic theorem.
			\begin{theorem}[Bergelson-Richter, {\cite[Corollary 1.28 and Lemma 6.3]{BergelsonRichter}}]
				\label{thm:beg-omega-nilsequence}
				Let \((Y,R)\) be a uniquely ergodic system with the
				unique $R$-invariant Borel probability measure \(\nu\).  Let \({\bf b}=(b_n)_{n\ge1}\)
				be a nilsequence.
				Then, for every \(G\in C(Y)\) and every \(y\in Y\),
				\[
				\lim_{N\to\infty}\frac1N\sum_{n=1}^N b_n\,G(R^{\Omega(n)}y)=
				\int_YG\,d\nu \left(\lim_{N\to\infty}\frac1N\sum_{n=1}^N b_n\right).
				\]
			\end{theorem}
			Next, we apply \Cref{thm:beg-omega-nilsequence} to derive the following.
			\begin{lemma}
				\label{lem:xiao-nilsequence-family}
				Let \((X,\mathcal X,\mu,S)\) be a $\Z$-system.
				Let \(b:X\times\mathbb N\to\mathbb C\) be such that \(x\mapsto b(x,n)\)
				is measurable for every \(n\), \(\sup_n\|b(\cdot,n)\|_{L^\infty(\mu)}<\infty\),
				and \(b_x:=(b(x,n))_{n\ge1}\) is a nilsequence for \(\mu\)-a.e.\ \(x\in X\).
				Then
				$\displaystyle M_b(x):=\lim_{N\to\infty} N^{-1}\sum_{n=1}^N b(x,n)$ exists for \(\mu\)-a.e. \(x\in X\),
				is measurable, and
				for every \(g\in L^\infty(\mu)\),
				\[
				\frac1N\sum_{n=1}^N b(x,n)S^{\Omega(n)}g(x)
				-
				M_b(x)\E_{\mu}(g\mid\mathcal I(S))(x)
				\to0
				\]
				in \(L^2(\mu)\) as $N\to\infty$.
			\end{lemma}
			\begin{proof}
				Set \(B=\sup_n\|b(\cdot,n)\|_{L^\infty(\mu)}\). 
				By the assumption, there is a co-null set $X_0\subset X$ such that for any $x\in X_0$, \(|b(x,n)|\leq B\) for all \(n\in \N\) and 
				\(b_x=(b(x,n))_{n\geq1}\) is a nilsequence.  
				By the last sentence in \Cref{subsec:nilpotent-preliminaries}, \(M_b(x)\) exists for all $x\in X_0$.  Moreover, since \(M_b\) is the 
				limit of the measurable function \(N^{-1}\sum_{n=1}^{N}b(\cdot,n)\), it is measurable.
				
				For \(f\in L^2(\mu)\), write
				\[
				A_N f(x)=\frac1N\sum_{n=1}^N b(x,n)f(S^{\Omega(n)}x).
				\]
				First assume that \(S\) is ergodic.  By the Jewett--Krieger theorem
				(see \cite{Jewett,Krieger}), we choose a uniquely ergodic topological model
				\((Y,\nu,R)\) of \((X,\mu,S)\), with the measure-theoretic isomorphism
				\(\phi:X\to Y\).  If \(H\in C(Y)\) and \(h=H\circ\phi\), then for $\mu$-a.e. \(x\in X_0\),
				using \Cref{thm:beg-omega-nilsequence} with the nilsequence \(b_x\), the point
				\(y=\phi(x)\), and \(G=H\), we get
				\[
				A_Nh(x)
				=
				\frac1N\sum_{n=1}^N b_x(n)H(R^{\Omega(n)}\phi(x))
				\longrightarrow
				M_b(x)\int_Y H\,d\nu
				=
				M_b(x)\int_X h\,d\mu
				\] as $N\to\infty$.
				The functions of the form
				\(H\circ\phi\), \(H\in C(Y)\), are dense in \(L^2(\mu)\).  Also
				\(\|A_N f\|_{L^2(\mu)}\leq B\|f\|_{L^2(\mu)}\) and
				$\displaystyle \norm{M_b\int_{X} f\,d\mu}_{L^2(\mu)}\leq B\|f\|_{L^2(\mu)}$.  Hence, for
				every \(g\in L^\infty(\mu)\),
				$$A_Ng(x)\longrightarrow
				M_b(x)\int_X g\,d\mu$$ in \(L^2(\mu)\) as $N\to\infty$.
				
				For the general case, let \(\displaystyle\mu=\int_Z\mu_z\,d\eta(z)\) be the ergodic
				decomposition of \(\mu\) with respect to $S$.  For \(\eta\)-a.e. \(z\in Z\), 
				\((X,\mathcal X,\mu_z,S)\) is ergodic and then apply the above argument to
				\(b\) and \(g\) on this ergodic component.  Thus
				\[
				\int_X
				\left|
				A_Ng(x)-M_b(x)\int_X g\,d\mu_z
				\right|^2
				\,d\mu_z(x)
				\longrightarrow 0
				\] as $N\to\infty$
				for \(\eta\)-a.e. \(z\in Z\). Note that 
				\(\displaystyle\E_{\mu}(g\mid\mathcal I(S))(x)=\int_X g\,d\mu_z\) for \(\mu_z\)-a.e. \(x\in X\). Then 
				\[
				\left\|
				A_Ng-M_b\,\E_{\mu}(g\mid\mathcal I(S))
				\right\|_{L^2(\mu)}^{2}=\int_Z\int_X
				\left|
				A_Ng(x)-M_b(x)\int_X g\,d\mu_z
				\right|^2
				\,d\mu_z(x)d\eta(z).
				\]
				The above equality and the Dominated Convergence Theorem finish the proof.
			\end{proof}
			\section{Cyclic tower actions}\label{Sec6}
			In this section, first, we introduce the cyclic tower action over a given $\Z$-system by the following, which is from the coding process of isotropy factors.
			
			Fix a $\Z$-system $(X,\X,\mu,U)$. Fix the tower height $L\ge 1$. Endow $Y:=X\times (\Z/L\Z)$ with the product $\nu$ of \(\mu\) and normalized counting measure. Define $V:Y\to Y$ by 
			\[
			(x,t)\mapsto
			\begin{cases}
				(x,t+1),& t\neq L-1,\\
				(Ux,0),& t=L-1.
			\end{cases}
			\]
			Then the inverse $V^{-1}$ of $V$ is the map 
			\[
			(x,t)\mapsto
			\begin{cases}
				(x,t-1),& t\neq 0,\\
				(U^{-1}x,L-1),& t=0.
			\end{cases}
			\]
			Moreover, for any measurable $A\times B\subset Y$,
			$$\nu(V^{-1}(A\times B))=
			\begin{cases}
				\frac{\mu(U^{-1}A)+(|B|-1)\mu(A)}{L},& 0\in B,\\
				\frac{|B|\mu(A)}{L}& 0\notin B.
			\end{cases}
			$$
			To sum up, $(Y,\nu,V)$ is a $\Z$-system. And if \(k=Lq+t\),
			\(q\in\mathbb Z\), \(0\leq t<L\), then
			\begin{equation}\label{eq6-1}
				V^k(x,0)=(U^qx,t).
			\end{equation}
			
			{\bf In the rest of the section, we will work at two systems $(X,\mu,U)$ and $(Y,\nu,V)$ all the time. So, in the next results, we omit them in the statements.}
			\subsection{Two lemmas}
			Next, we will present two lemmas for cyclic tower actions. They display invariant factors and $Z_{s}(V^L)$, respectively.
			\begin{lemma}
				\label{lem:cyclic-tower-invariant-factor}
				Suppose that \(S:X\to X\) is an invertible
				measure preserving transformation commuting with \(U\), and define
				$
				\widetilde S(x,t)=(Sx,t).
				$
				Then \(\widetilde S\) is invertible and measure preserving,
				$
				\widetilde S V=V\widetilde S,
				$
				and
				$
				\mathcal I(\widetilde S)
				=
				\mathcal I(S)\otimes\mathcal P(\mathbb Z/L\mathbb Z).
				$
				And for \(g\in L^\infty(\mu)\) and \(\widetilde g(x,t)=g(x)\),
				$$
				\E_{\nu}(\widetilde g\mid\mathcal I(\widetilde S))(x,t)
				=
				\E_{\mu}(g\mid\mathcal I(S))(x).
				$$ Furthermore, for \(h\in L^\infty(\nu)\) and \( h_{t}(x)=h(x,t)\),
				$$
				\E_{\nu}(h\mid\mathcal I(\widetilde S))(x,t)
				=
				\E_{\mu}(h_t\mid\mathcal I(S))(x).
				$$
			\end{lemma}
			\begin{proof}
				The setting of \(S\) gives that  \(\widetilde S\) is invertible and measure preserving.  \(\widetilde SV=V\widetilde S\) follows from \(SU=US\).
				For \(F\in L^2(\nu)\), write \(F_t(x)=F(x,t)\).  Then
				\(\widetilde SF=F\) if and only if \(SF_t=F_t\) for every
				\(t\in\mathbb Z/L\mathbb Z\).  So, $
				\mathcal I(\widetilde S)
				=
				\mathcal I(S)\otimes\mathcal P(\mathbb Z/L\mathbb Z).
				$ By Birkhoff's ergodic theorem, for $\nu$-a.e. $(x,t)\in Y$,
				$$\E_{\nu}(\widetilde g\mid\mathcal I(\widetilde S))(x,t)=\lim_{N\to\infty}\frac{1}{N}\sum_{n=1}^{N}\widetilde g(S^n x,t)=\lim_{N\to\infty}\frac{1}{N}\sum_{n=1}^{N} g(S^n x)=\E_{\mu}(g\mid\mathcal I(S))(x);$$
				$$
				\E_{\nu}(h\mid\mathcal I(\widetilde S))(x,t)=\lim_{N\to\infty}\frac{1}{N}\sum_{n=1}^{N}h(S^n x,t)=\lim_{N\to\infty}\frac{1}{N}\sum_{n=1}^{N}h_t(S^n x)	=
				\E_{\mu}(h_t\mid\mathcal I(S))(x).
				$$ 
				This finishes the proof.
			\end{proof}
			\begin{lemma}
				\label{lem:cyclic-tower-host-kra}
				For \(H\in L^\infty(\nu)\), write \(H_t(x)=H(x,t)\).  Then, for every
				\(m\geq1\),
				\begin{equation}\label{eq:cyclic-tower-power-seminorm}
					\|H\|_{U^m(V^L)}^{2^m}
					=\frac1L\sum_{t\in\mathbb Z/L\mathbb Z}
					\|H_t\|_{U^m(U)}^{2^m}.
				\end{equation}
				Consequently, for every \(s\geq1\),
				$
				Z_s(V^L)=Z_s(U)\otimes\mathcal P(\mathbb Z/L\mathbb Z)
				\subseteq Z_s(V).
				$
			\end{lemma}
			
			\begin{proof}
				Let
				$
				\displaystyle\mu=\int_{\Omega}\mu_\omega\,d\lambda(\omega)
				$
				be the ergodic decomposition of $\mu$ with respect to $U$.  Since
				\(V^L(x,t)=(Ux,t)\), the ergodic decomposition of \((Y,\nu,V^L)\) can be written as
				\[
				\nu=\frac1L\sum_{t\in\mathbb Z/L\mathbb Z}
				\int_{\Omega}(\mu_\omega\otimes\delta_t)\,d\lambda(\omega).
				\]
				Applying \Cref{lem:HK-seminorm-ergodic-decomposition} first to
				\((Y,\nu,V^L)\) and then to \((X,\mu,U)\) gives
				\[
				\begin{aligned}
					\|H\|_{U^m(V^L)}^{2^m}
					=\frac1L\sum_{t\in\mathbb Z/L\mathbb Z}\int_\Omega
					\|H_t\|_{U^m(X,\mu_\omega,U)}^{2^m}
					\,d\lambda(\omega)
					=\frac1L\sum_{t\in\mathbb Z/L\mathbb Z}\|H_t\|_{U^m(U)}^{2^m},
				\end{aligned}
				\]
				which is \eqref{eq:cyclic-tower-power-seminorm}.
				Clearly, \eqref{eq:cyclic-tower-power-seminorm} gives $
				Z_s(V^L)=Z_s(U)\otimes\mathcal P(\mathbb Z/L\mathbb Z).
				$
				On the other hand, \Cref{lem:one-dimensional-power-comparison}, applied to
				\(V\) and its power \(V^L\), gives
				\[
				\|H\|_{U^{s+1}(V^L)}
				\leq C_{L,s}\|H\|_{U^{s+1}(V)}.
				\]
				This means that \(Z_s(V^L)\subseteq Z_s(V)\). The proof is complete.
			\end{proof}
			\subsection{One-action decoupling on a cyclic tower action}
			In this subsection, we give two decoupling results involving $\Omega(n)$ for cyclic tower actions.
			\begin{proposition}
				\label{prop:cyclic-tower-multislope}
				Let \(c:X\times\mathbb Z\to\mathbb C\) be a bounded measurable weight, which is Eisner--Zorin-Kranich regular with respect to \(U\).
				Put
				$
				\widetilde c((x,t),n)=c(x,n).
				$
				Let \(A\subset\mathbb Z\setminus\{0\}\) be finite and non-empty, and let $1$-bounded
				\(H_a\in L^\infty(\nu),a\in A\).  If
				$
				\E_{\nu}(H_{a_0}\mid Z_\infty(V))=0
				$
				for some \(a_0\in A\), then
				\[
				\frac1N\sum_{n=1}^N
				\widetilde c((x,t),n)\prod_{a\in A}V^{an}H_a
				\longrightarrow0
				\]
				in \(L^2(\nu)\) as $N\to\infty$.
			\end{proposition}
			
			\begin{proof}
				Fix \(t\in\mathbb Z/L\mathbb Z\) and
				\(r\in\{0,\ldots,L-1\}\).  For every \(a\in A\), write uniquely
				$
				t+ar=L\ell_{a,r,t}+\tau_{a,r,t},
				\ell_{a,r,t}\in\mathbb Z,0\leq\tau_{a,r,t}<L.
				$
				If \(n=Lm+r\), then by \eqref{eq6-1},
				$$
				(V^{an}H_a)(x,t)
				=H_a\bigl(U^{am+\ell_{a,r,t}}x,\tau_{a,r,t}\bigr)\\
				=U^{am}K_{a,r,t}(x),
				$$
				where
				$
				K_{a,r,t}=U^{\ell_{a,r,t}}H_{a,\tau_{a,r,t}}.
				$
				By \Cref{lem:cyclic-tower-host-kra},
				\(H_{a_0,\tau_{a_0,r,t}}\) is orthogonal to \(Z_\infty(U)\). 
				Then \(K_{a_0,r,t}\) is orthogonal to \(Z_\infty(U)\).
				
				Set
				$
				c_r(x,m)=c(x,Lm+r).
				$ By \Cref{def:finite-ezk-stability}, $c_r(x,m)$ is Eisner--Zorin-Kranich regular with respect to $U$.
				Then \Cref{prop:multislope-nilsequence-estimate} gives
				\[
				\frac1M\sum_{m=1}^M
				c_r(x,m)\prod_{a\in A}U^{am}K_{a,r,t}
				\longrightarrow0
				\]
				in \(L^2(\mu)\) as $N\to\infty$.  
				
				Split the original average into the finitely many residue classes
				\(n\equiv r\pmod L\).  If \(M_{N,r}\) denotes the number of integers
				\(n\leq N\) in the residue class \(r\), the corresponding block is
				\[
				\frac{M_{N,r}}{N}
				\left(
				\frac1{M_{N,r}}\sum_{m}
				c_r(x,m)\prod_{a\in A}U^{am}K_{a,r,t}
				\right),
				\]
				up to at most one term.  Since \(M_{N,r}/N\leq1\), each block
				tends to zero in \(L^2(\mu)\).  Summing over \(r\) and then averaging over \(t\) proves the proposition.
			\end{proof}
			\begin{proposition}
				\label{prop:cyclic-tower-one-action-decoupling}
				Let $S$ and $\widetilde{S}$ be defined as \Cref{lem:cyclic-tower-invariant-factor}.
				Let $c$ and $\widetilde{c}$ be defined as \Cref{prop:cyclic-tower-multislope}.
				Then for integer affine forms \(L_i(n):=a_i n+r_i\), and all 
				\(h_1,\ldots,h_m,g\in L^\infty(\nu)\), one has
				\begin{align*}
					\frac1N\sum_{n=1}^N
					\widetilde c((x,t),n)
					\Bigl(\prod_{i=1}^mV^{L_i(n)}h_i\Bigr)
					\widetilde S^{\Omega(n)}g-
					\left(\frac1N\sum_{n=1}^N
					\widetilde c((x,t),n)
					\prod_{i=1}^mV^{L_i(n)}h_i\right)
					\E_{\nu}(g\mid\mathcal I(\widetilde S))
					\longrightarrow0
				\end{align*}
				in \(L^2(\nu)\) as $N\to\infty$.
			\end{proposition}
			\begin{proof}
				Replacing \(h_i\) by \(V^{r_i}h_i\), we may assume \(r_i=0\) for all \(i\).
				We may also assume \(a_i\neq0\) for all \(i\). Indeed, the terms with
				\(a_i=0\) give the bounded \(H_0:=\prod_{1\le i\le m,a_i=0}h_i\), independent of
				\(n\). If all \(a_i=0\), applying \Cref{def:finite-ezk-stability} with
				\(R=1\), \(\epsilon_1=0\), \((r_1,t_1,a_1)=(0,0,1)\), and \(h_1=0\), gives a finite Eisner--Zorin-Kranich representation of
				\(c(x,n)\). So, for $\mu$-a.e. $x\in X$, $c(x,n)$ is a nilsequence.  Write \(g_t(x)=g(x,t)\). Apply \Cref{lem:xiao-nilsequence-family} and then
				\[
				\frac1N\sum_{n=1}^Nc(x,n)S^{\Omega(n)}g_t
				-
				\left(\frac1N\sum_{n=1}^Nc(x,n)\right)
				\E_{\mu}(g_t\mid\mathcal I(S))
				\longrightarrow0
				\]
				in \(L^2(\mu)\) as $N\to\infty$. Note that $S^{\Omega(n)}g_t(x)=\widetilde S^{\Omega(n)}g(x,t)$. By \Cref{lem:cyclic-tower-invariant-factor},
				$$
				\frac1N\sum_{n=1}^N
				\widetilde c((x,t),n)
				\widetilde S^{\Omega(n)}g-
				\left(\frac1N\sum_{n=1}^N
				\widetilde c((x,t),n)\right)
				\E_{\nu}(g\mid\mathcal I(\widetilde S))
				\longrightarrow0
				$$
				in \(L^2(\nu)\) as $N\to\infty$; multiplication by \(H_0\) proves this case. In the mixed
				case, the difference is \(H_0\) times the
				corresponding difference after deleting the indices with \(a_i=0\).
				Thus it remains to
				consider finite $A\subset\mathbb Z\setminus\{0\}$.
				
				For each $a\in A$, let $H_a=\prod_{1\le i\le m,a_i=a}h_i$.  We may suppose
				\(\|H_a\|_{L^{\infty}(\nu)}\leq1\).
				Suppose first that
				\(\E_{\nu}(H_{a_0}\mid Z_\infty(V))=0\) for some \(a_0\in A\).
				Then \Cref{prop:cyclic-tower-multislope} gives
				\[
				\frac1N\sum_{n=1}^N
				\widetilde c((x,t),n)\prod_{a\in A}V^{an}H_a
				\longrightarrow0
				\]
				in $L^2(\nu)$ as $N\to\infty$.
				By the linear property of ergodic averages, we may assume that \(|g|\equiv1\). Set
				\[
				v_n=\widetilde c((x,t),n)
				\prod_{a\in A}V^{an}H_a\,
				\widetilde S^{\Omega(n)}g.
				\]
				For distinct primes \(p,q\), 
				\[
				\langle v_{pn},v_{qn}\rangle=\int_Y
				\widetilde c((x,t),pn)
				\overline{\widetilde c((x,t),qn)}
				\prod_{a\in A}V^{apn}H_a\,V^{aqn}\overline{H_a}\,d\nu.
				\]
				Note that 
				$
				\widetilde c((x,t),pn)
				\overline{\widetilde c((x,t),qn)}=c(x,pn)\overline{c(x,qn)}.
				$ By \Cref{def:finite-ezk-stability}, 
				$c(x,pn)\overline{c(x,qn)}$ is Eisner--Zorin-Kranich regular with respect to \(U\). 
				Only finitely many distinct prime pairs satisfy
				\(ap=a'q\) for some \(a,a'\in A\).  Delete all primes occurring in those
				pairs.  
				Then for any two remaining primes $p,q$, applying
				\Cref{prop:cyclic-tower-multislope} gives
				$$
				\frac1N\sum_{n=1}^N\widetilde c((x,t),pn)
				\overline{\widetilde c((x,t),qn)}
				\prod_{a\in A}V^{apn}H_a\,V^{aqn}\overline{H_a}
				\longrightarrow 0
				$$
				in $L^2(\nu)$ as $N\to\infty$.
				Hence,
				\Cref{lem:orthogonality} gives
				\[
				\frac1N\sum_{n=1}^N
				\widetilde c((x,t),n)\prod_{a\in A}V^{an}H_a\,
				\widetilde S^{\Omega(n)}g
				\longrightarrow0
				\]
				in $L^2(\nu)$ as $N\to\infty$.
				
				It remains to treat \(H_a\in L^\infty(Y,Z_\infty(V),\nu)\) for every
				\(a\in A\). Fix sufficiently small \(\eta>0\). Note that 
				\(Z_\infty(V)=\bigvee_{k\geq1}Z_k(V)\) and by \cite[Chapter 8, Lemma 12]{HostKraBook} and the definition of $Z_k(V)$, $Z_k(V)$ is increasing. So, the Martingale Theorem gives \(k\geq1\) such that for all $a\in A$,
				$
				\|H_a-\E_{\nu}(H_a\mid Z_k(V))\|_{L^{2}(\nu)}<\eta.
				$
				Applying \Cref{lem:smooth-bundle-approx} to \((Y,\nu,V)\) and
				\(\{\E_{\nu}(H_a\mid Z_k(V)):a\in A\}\) gives tame
				nilfunctions \(K_a,a\in A\) such that 
				\(\|H_a-K_a\|_{L^{2}(\nu)}<\eta\) and $\norm{K_a}_{{L^{\infty}(\nu)}}\le 1+\eta$.  
				
				As mentioned before, for $\mu$-a.e. $x\in X$, \(c(x,n)\) is a nilsequence. So, for $\nu$-a.e. $(x,t)\in Y$,
				\(\widetilde c((x,t),n)\) is a nilsequence. For every \(a\in A\), by \Cref{def:tame-finite-stage-nilfunction}, for $\nu$-a.e. $y\in Y$,
				\(n\mapsto K_a(V^{an}y)\) is a nilsequence. Note that 
				any finite product of
				nilsequences is still a nilsequence. So, $\nu$-a.e. $(x,t)\in Y$,
				\[
				d((x,t),n):=\widetilde c((x,t),n)
				\prod_{a\in A}K_a(V^{an}(x,t))
				\]
				is a uniformly bounded nilsequence.  Applying \Cref{lem:xiao-nilsequence-family} to $\widetilde{S}$ and \(d\) gives
				\begin{equation}\label{eq2222}
					\frac1N\sum_{n=1}^N d((x,t),n)	\widetilde S^{\Omega(n)}g
					-
					\left(\frac1N\sum_{n=1}^N d((x,t),n)\right)\E_{\nu}(g\mid\mathcal I(\widetilde S))(x,t)
					\to0
				\end{equation}
				in \(L^2(\nu)\) as $N\to\infty$.
				On the other hand, note that the following estimate:
				\begin{equation}\label{eq3333}
					\begin{split}
						&\norm{\frac1N\sum_{n=1}^N d((x,t),n)\widetilde S^{\Omega(n)}g(x,t)-\frac1N\sum_{n=1}^N \widetilde c((x,t),n)\prod_{a\in A}V^{an}H_a(x,t)\widetilde S^{\Omega(n)}g(x,t)}_{L^2(\nu)}
						\\ & \hspace{0.5cm} \le \norm{g}_{L^{\infty}(\nu)}\sup_{n}\norm{c(x,n)}_{L^{\infty}_x(\mu)}\frac1N\sum_{n=1}^N\norm{\prod_{a\in A}H_a(V^{an}(x,t))-\prod_{a\in A}K_a(V^{an}(x,t))}_{L^2(\nu)}.
					\end{split}
				\end{equation}
				
				and
				
				\begin{equation}\label{eq4444}
					\begin{split}
						&\norm{\frac1N\sum_{n=1}^N d((x,t),n)\E_{\nu}(g\mid\mathcal I(\widetilde S))-\frac1N\sum_{n=1}^N \widetilde c((x,t),n)\prod_{a\in A}V^{an}H_a(x,t)\E_{\nu}(g\mid\mathcal I(\widetilde S))}_{L^2(\nu)}
						\\ & \hspace{0.5cm} \le \norm{g}_{L^{\infty}(\nu)}\sup_{n}\norm{c(x,n)}_{L^{\infty}_x(\mu)}\frac1N\sum_{n=1}^N\norm{\prod_{a\in A}H_a(V^{an}(x,t))-\prod_{a\in A}K_a(V^{an}(x,t))}_{L^2(\nu)}.
					\end{split}
				\end{equation}
				Clearly, \eqref{eq2222}, \eqref{eq3333} and \eqref{eq4444} finish the proof.
			\end{proof}
			\section{Reduction of \Cref{thm:main-decoupling}: finite structures}
			\label{sec:isotropy-completion}
			
			This section returns to the characteristic factors obtained in
			\Cref{sec:orthogonality-austin}.  The finite isotropy factors are coded, on residue
			classes, by affine iterates of a single transformation \(U=T^v\).  The nilfactor
			part becomes an Eisner--Zorin-Kranich regular weight, so the one action theorem from
			\Cref{Sec6} applies.  
			
			\subsection{Coding isotropy factors}
			\begin{lemma}
				\label{lem:isotropy-coding}
				Let \((X,\mathcal X,\mu,(T^v)_{v\in \Z^2})\) be a \(\mathbb Z^2\)-system.  Let
				\(w\in\mathbb Z^2\setminus\{0\}\), let \(v,e\in\mathbb Z^2\), and assume
				\(\det(v,w)\neq0\).  Put \(U=T^v\),
				$
				D_w=\det(v,w),A_{e,w}=\det(e,w), L_w=|D_w|.
				$
				If \(h\in L^\infty(X,\mathcal I(T^w),\mu)\), then for every
				\(r\in\{0,\ldots,L_w-1\}\) and every \(n\equiv r\pmod{L_w}\),
				$
				T^{ne}h
				=
				U^{P_{e,w,r}(n)}h_{e,w,r},
				$
				where
				$
				h_{e,w,r}=T^{re}h,
				P_{e,w,r}(n)=
				\operatorname{sgn}(D_w)A_{e,w}\frac{n-r}{L_w}.
				$
			\end{lemma}
			\begin{proof}
				$
				\det(v,w)e-\det(e,w)v=\det(v,e)w
				$
				gives, since \(h\) is \(T^w\)-invariant,
				\[
				T^{D_we}h=T^{A_{e,w}v}h=U^{A_{e,w}}h .
				\]
				Thus, for every \(m\in\mathbb Z\),
				$
				T^{L_wme}h
				=
				U^{\operatorname{sgn}(D_w)A_{e,w}m}h .
				$
				Writing \(n=L_wm+r\), we obtain
				\[
				T^{ne}h
				=
				T^{re}T^{L_wme}h
				=
				U^{\operatorname{sgn}(D_w)A_{e,w}m}T^{re}h
				=
				U^{P_{e,w,r}(n)}h_{e,w,r}.
				\]
				This finishes the proof.
			\end{proof}

			\subsection{Finite structured decoupling}
			First, we introduce a class of bounded functions, which will be used to approximate the characteristic factors obtained in
			\Cref{sec:orthogonality-austin}.
			\begin{definition}
				\label{def:finite-structured-algebra}
				\begin{samepage}
					Let \((X,\mathcal X,\mu,(T^v)_{v\in \Z^2})\) be a \(\mathbb Z^2\)-system.  An
					\emph{elementary finite structured function} is a finite product shaped like
					$
					\displaystyle\left(\prod_{\alpha\in I}h_\alpha\right)
					\left(\prod_{\beta\in J}\phi_\beta\right),
					$
					where \(I,J\) are finite, empty products are allowed,
					$
					h_\alpha\in L^\infty(X,\mathcal I(T^{w_\alpha}),\mu),
					w_\alpha\in\mathbb Z^2\setminus\{0\},
					$
					and each \(\phi_\beta\) is a tame $2$-step nilfunction.
					Let \(\mathscr A_{\rm str}(X,\mathcal X,\mu,(T^v)_{v\in \Z^2})\) be the finite linear span of all elementary
					finite structured functions.
				\end{samepage}
			\end{definition}
			Next, we build a decoupling result for the family $\mathscr A_{\rm str}$.
			\begin{proposition}
				\label{prop:finite-nil-decoupling}
				Let \((X,\mathcal X,\mu,(T^v)_{v\in \Z^2})\) be a \(\mathbb Z^2\)-system, \(S:X\to X\) be an
				invertible measure preserving transformation  commuting with \(T\), and put
				\(e_1=(1,0)\), \(e_2=(0,1)\).  If
				$
				F_i\in\mathscr A_{\rm str}(X,\mathcal X,\mu,(T^v)_{v\in \Z^2}),i=1,2,
				$
				then, for every \(g\in L^\infty(\mu)\),
				\[
				\frac1N\sum_{n=1}^{N}
				T^{ne_1}F_1\,T^{ne_2}F_2\,S^{\Omega(n)}g
				-
				\left(
				\frac1N\sum_{n=1}^{N}
				T^{ne_1}F_1\,T^{ne_2}F_2
				\right)
				\E_{\mu}(g\mid\mathcal I(S))
				\to0
				\]
				in \(L^2(\mu)\) as $N\to\infty$.
			\end{proposition}
			
			\begin{proof}
				By the definition of \(\mathscr A_{\rm str}(X,T)\), it is enough
				to prove the assertion when
				\[
				F_i=
				\left(\prod_{\alpha\in I_i}h_{i,\alpha}\right)
				\left(\prod_{\beta\in J_i}\phi_{i,\beta}\right),
				i=1,2,
				\]
				is an elementary finite structured function, where 
				$
				h_{i,\alpha}\in L^\infty(X,\mathcal I(T^{w_{i,\alpha}}),\mu),
				w_{i,\alpha}\in\mathbb Z^2\setminus\{0\},
				$
				and each \(\phi_{i,\beta}\) is a tame $2$-step nilfunction.
				
				The rest of the proof is divided into four steps.
				
				\medskip
				
				\noindent\emph{Step 1: Choose one transverse action and isolate the nilfunction weight.}
				
				\medskip
				
				Let
				\(
				W=\{w_{i,\alpha}: i=1,2,\ \alpha\in I_i\}.
				\)
				If \(W=\emptyset\), choose any non-zero \(v\in\mathbb Z^2\).  If
				\(W\neq\emptyset\), choose \(v\in\mathbb Z^2\) such that for each $w\in W$, 
				$
				\det(v,w)\neq0.
				$
				Put \(U=T^v\).
				
				Let
				$
				\mathcal Q=\{(i,\beta):i=1,2,\ \beta\in J_i\}.
				$
				Define
				\[
				w(x,n)=
				\prod_{(i,\beta)\in\mathcal Q}
				\bigl(T^{ne_i}\phi_{i,\beta}\bigr)(x).
				\]
				If \(\mathcal Q=\emptyset\), then \(w(x,n)\equiv 1\), and \(w\) is Eisner--Zorin-Kranich regular with respect to \(U\).  If \(\mathcal Q\neq
				\emptyset\), enumerate \(\mathcal Q=\{q_1,\ldots,q_Q\}\), write
				\(q_j=(i_j,\beta_j)\), and set
				$
				u_j=e_{i_j},
				\psi_j=\phi_{i_j,\beta_j},
				1\leq j\leq Q .
				$
				By \Cref{lem:bundle-weight-ezk-stable}, applied to
				\((\psi_j,u_j)_{j=1}^{Q}\) and \(U=T^v\), \(w\) is
				Eisner--Zorin-Kranich regular with respect to \(U\).  
				
				\medskip
				
				\noindent\emph{Step 2: Code the finite isotropy part on residue classes.}
				
				\medskip
				
				If \(W=\emptyset\), put \(L=1\).  If
				\(W\neq\emptyset\), put
				\[
				L=\operatorname{lcm}\{|\det(v,w_{i,\alpha})|:
				i=1,2,\ \alpha\in I_i\}.
				\]
				Fix \(r\in\{0,\ldots,L-1\}\).  For each pair \((i,\alpha)\), put
				$
				D_{i,\alpha}=\det(v,w_{i,\alpha}),
				L_{i,\alpha}=|D_{i,\alpha}|,
				$
				and let
				$
				r_{i,\alpha}\in\{0,\ldots,L_{i,\alpha}-1\}
				$
				be the residue class of \(r\) modulo \(L_{i,\alpha}\).  Since
				\(L_{i,\alpha}\mid L\), every \(n\equiv r\pmod L\) satisfies
				$
				n\equiv r_{i,\alpha}\pmod {L_{i,\alpha}}.
				$
				Applying \Cref{lem:isotropy-coding} with \(e=e_i\) and
				\(w=w_{i,\alpha}\), for all $n\equiv r\pmod L$,
				$
				T^{ne_i}h_{i,\alpha}
				=
				U^{P_{i,\alpha,r}(n)}a_{i,\alpha,r},
				$
				where
				$
				a_{i,\alpha,r}=T^{r_{i,\alpha}e_i}h_{i,\alpha},
				$
				and
				$
				P_{i,\alpha,r}(n)
				=
				\operatorname{sgn}(D_{i,\alpha})
				\det(e_i,w_{i,\alpha})
				\frac{n-r_{i,\alpha}}{L_{i,\alpha}}.
				$
				Since
				$
				n-r_{i,\alpha}
				=
				(n-r)+(r-r_{i,\alpha}),
				$
				we can write
				$
				P_{i,\alpha,r}(n)
				=
				A_{i,\alpha,r}\frac{n-r}{L}
				+
				B_{i,\alpha,r},
				$
				where
				$
				A_{i,\alpha,r}
				=
				\operatorname{sgn}(D_{i,\alpha})
				\det(e_i,w_{i,\alpha})
				\frac{L}{L_{i,\alpha}}
				\in\mathbb Z,
				$
				and
				$
				B_{i,\alpha,r}
				=
				\operatorname{sgn}(D_{i,\alpha})
				\det(e_i,w_{i,\alpha})
				\frac{r-r_{i,\alpha}}{L_{i,\alpha}}
				\in\mathbb Z.
				$
				Enumerating the finitely many pairs \((i,\alpha)\) as
				\(j=1,\ldots,J_r\), we obtain \(a_{j,r}\in L^\infty(\mu)\) and
				\(A_{j,r},B_{j,r}\in \Z\) such that for all $n\equiv r\pmod L$,
				\[
				\prod_{i=1,2}\prod_{\alpha\in I_i}
				T^{ne_i}h_{i,\alpha}
				=
				\prod_{j=1}^{J_r}
				U^{A_{j,r}(n-r)/L+B_{j,r}}a_{j,r}.
				\]
				
				\medskip
				
				\noindent\emph{Step 3: Remove the residue restriction by a cyclic tower action.}
				
				\medskip
				
				Fix a residue class \(r\). Define the cyclic tower action $(Y,\nu,V)$ with the height $L$ over $(X,\mu,U)$ as the beginning of \Cref{Sec6}. Define $\widetilde S:Y\to Y,(x,t)\mapsto (Sx,t)$. Define $\widetilde g(x,t)=g(x)$ and $\widetilde w((x,t),n)=w(x,n)$.
				Also put
				$
				\chi_0(x,t)=\mathbf 1_{\{0\}}(t),
				\widetilde a_{j,r}(x,t)=\mathbf 1_{\{0\}}(t)a_{j,r}(x).
				$
				Let 
				$
				Q_{0,r}(n)=n-r,
				$
				and, for \(1\leq j\leq J_r\), put
				$
				Q_{j,r}(n)=A_{j,r}(n-r)+LB_{j,r}.
				$
				Define
				\[
				\mathcal H_r(n,x)=
				\begin{cases}
					\displaystyle
					\prod_{j=1}^{J_r}
					U^{A_{j,r}(n-r)/L+B_{j,r}}a_{j,r}(x),
					& n\equiv r\pmod L,\\[1ex]
					0,& n\not\equiv r\pmod L.
				\end{cases}
				\]
				By \eqref{eq6-1},
				$$
				\mathcal H_r(n,x)
				=
				\bigl(V^{Q_{0,r}(n)}\chi_0\bigr)(x,0)
				\prod_{j=1}^{J_r}
				\bigl(V^{Q_{j,r}(n)}\widetilde a_{j,r}\bigr)(x,0).
				$$
				
				\medskip
				
				\noindent\emph{Step 4: Apply the one action decoupling theorem.}
				
				\medskip
				
				Put 
				\begin{align*}
					D_{N,r}(x,t)
					=&
					\frac1N\sum_{n=1}^{N}
					\widetilde w((x,t),n)
					\bigl(V^{Q_{0,r}(n)}\chi_0\bigr)
					\prod_{j=1}^{J_r}
					\bigl(V^{Q_{j,r}(n)}\widetilde a_{j,r}\bigr)
					\widetilde S^{\Omega(n)}\widetilde g                                      \\
					& \hspace{1cm} -
					\left(
					\frac1N\sum_{n=1}^{N}
					\widetilde w((x,t),n)
					\bigl(V^{Q_{0,r}(n)}\chi_0\bigr)
					\prod_{j=1}^{J_r}
					\bigl(V^{Q_{j,r}(n)}\widetilde a_{j,r}\bigr)
					\right)
					\E_{\nu}(\widetilde g\mid\mathcal I(\widetilde S)).
				\end{align*}
				
				Apply \Cref{prop:cyclic-tower-one-action-decoupling} to
				\((Y,\nu,V)\) with \(w\), \(S\),
				$
				Q_{0,r},Q_{1,r},\ldots,Q_{J_r,r},
				$
				and
				$
				\chi_0,\widetilde a_{1,r},\ldots,\widetilde a_{J_r,r}.
				$
				We obtain
				$
				D_{N,r}\to0
				\ \text{in }L^2(\nu).
				$
				Restricting to the fibre \(t=0\), \Cref{lem:cyclic-tower-invariant-factor}, and the preceding identities give
				\begin{align*}
					& \lim_{N\to\infty}\norm{\frac1N\sum_{n=1}^{N}
						w(x,n)\mathcal H_r(n,x)S^{\Omega(n)}g(x)                   
						-
						\left(
						\frac1N\sum_{n=1}^{N}
						w(x,n)\mathcal H_r(n,x)
						\right)
						\E_{\mu}(g\mid\mathcal I(S))(x)}_{L^{2}(\mu)}=0.
				\end{align*}
				
				Note that for every \(n\),
				$
				\displaystyle\prod_{i=1,2}\prod_{\alpha\in I_i}
				T^{ne_i}h_{i,\alpha}(x)
				=
				\sum_{r=0}^{L-1}\mathcal H_r(n,x).
				$
				Therefore, 
				\begin{align*}
					& \frac1N\sum_{n=1}^{N}
					w(x,n)\prod_{i=1,2}\prod_{\alpha\in I_i}
					T^{ne_i}h_{i,\alpha}(x)S^{\Omega(n)}g(x)               
					\\ & \hspace{3cm} -
					\left(
					\frac1N\sum_{n=1}^{N}
					w(x,n)\prod_{i=1,2}\prod_{\alpha\in I_i}
					T^{ne_i}h_{i,\alpha}(x)	\right)
					\E_{\mu}(g\mid\mathcal I(S))(x)\to 0
				\end{align*}
				in $L^{2}(\mu)$ as $N\to\infty$. This finishes the proof.
			\end{proof}

			\subsection{Passage to the characteristic factors}
			
			\begin{proposition}
				\label{prop:structured-decoupling}
				Let \(e_1=(1,0)\), \(e_2=(0,1)\). Let \((X,\mathcal X,\mu,(T^v)_{v\in \Z^2})\) be a \(\mathbb Z^2\)-system, and \(S:X\to X\) be
				an invertible measure-preserving transformation commuting with \(T\).  Let
				\(\mathcal N_{X}\subseteq Z_2(T)\) be a \(T\)-invariant factor. Let
				\(\mathcal Z_{1,\mathcal N_{X}},\mathcal Z_{2,\mathcal N_{X}}\) be two factors defined from \(\mathcal N_{X}\) in
				\eqref{eq:char-factors}.  If
				$
				f_i\in L^\infty(X,\mathcal Z_{i,\mathcal N_{X}},\mu),i=1,2,
				$
				then, for every \(g\in L^\infty(\mu)\),
				\[
				\frac1N\sum_{n=1}^{N}
				T^{ne_1}f_1\,T^{ne_2}f_2\,S^{\Omega(n)}g
				-
				\left(
				\frac1N\sum_{n=1}^{N}
				T^{ne_1}f_1\,T^{ne_2}f_2
				\right)
				\E_{\mu}(g\mid\mathcal I(S))
				\to0
				\]
				in \(L^2(\mu)\) as $N\to\infty$.
			\end{proposition}
			
			\begin{proof}
				We may assume that $\max(\|f_1\|_{L^{\infty}(\mu)},\|f_2\|_{L^{\infty}(\mu)},\|g\|_{L^{\infty}(\mu)})\le 1/4$.
				For \(i=1,2\), the factor \(\mathcal Z_{i,\mathcal N_{X}}\) is the join of \(\mathcal N_{X}\)
				with countably many isotropy factors.  Hence, by the Martingale Theorem,
				for any \(\varepsilon>0\), there are finite sets
				\(W_i\subset\mathbb Z^2\setminus\{0\}\) such that, with
				\[
				\mathcal Z_{i,\mathcal N_{X},W_i}
				=
				\mathcal N_{X}\vee\bigvee_{w\in W_i}\mathcal I(T^w)\ \text{and}\ f_{i,0}=\E_{\mu}(f_i\mid\mathcal Z_{i,\mathcal N_{X},W_i}),
				\]
				$
				\|f_i-f_{i,0}\|_{L^{2}(\mu)}<\varepsilon.
				$
				
				Then for each \(i=1,2\), we choose a
				sequence \(G_i^{(m)}\) of finite linear combinations of products, shaped like
				$
				H_i\prod_{w\in W_i}h_{i,w},
				H_i\in L^\infty(X,\mathcal N_{X},\mu),
				h_{i,w}\in L^\infty(X,\mathcal I(T^w),\mu),
				$
				such that
				$
				\|f_{i,0}-G_i^{(m)}\|_{L^{2}(\mu)}\to0,
				\|G_i^{(m)}\|_{L^{\infty}(\mu)}\leq 1.
				$
				Fix \(m\).  Only finitely many \(\mathcal N_{X}\)-measurable functions occur
				in the representation of $G_i^{(m)}$.  Since
				\(\mathcal N_{X}\subseteq Z_2(T)\), we apply
				\Cref{lem:smooth-bundle-approx} to precisely this finite family, with
				\(\Lambda=\mathbb Z^2\).  Replacing those functions by the resulting
				tame $2$-step nilfunctions produces
				\(G_{i}^{(m,\delta)}\in\mathscr A_{\rm str}(X,\mathcal X,\mu,(T^v)_{v\in \Z^2})\) such that
				$
				\|G_i^{(m,\delta)}-G_i^{(m)}\|_{L^{2}(\mu)}\to0
				$ as $\delta\downarrow0$.
				Note that the approximating functions supplied by
				\Cref{lem:smooth-bundle-approx} have uniformly bounded
				\(L^\infty(\mu)\)-norms for all \(0<\delta<1\). Then, 
				$
				\sup_{0<\delta<1}\|G_i^{(m,\delta)}\|_{L^\infty(\mu)}<\infty.
				$
				
				Applying \Cref{prop:finite-nil-decoupling} to \(G_1^{(m,\delta)},G_2^{(m,\delta)}\) gives
				\[
				\frac1N\sum_{n=1}^{N}
				T^{ne_1}G_1^{(m,\delta)}\,T^{ne_2}G_2^{(m,\delta)}\,S^{\Omega(n)}g
				-
				\left(
				\frac1N\sum_{n=1}^{N}
				T^{ne_1}G_1^{(m,\delta)}\,T^{ne_2}G_2^{(m,\delta)}
				\right)
				\E_{\mu}(g\mid\mathcal I(S))
				\to0
				\]
				in \(L^2(\mu)\) as $N\to\infty$. Since $
				\sup_{0<\delta<1}\|G_i^{(m,\delta)}\|_{L^\infty(\mu)}<\infty,
				$ the standard approximation argument gives
				\begin{equation}\label{eq7-1}
					\frac1N\sum_{n=1}^{N}
					T^{ne_1}G_1^{(m)}\,T^{ne_2}G_2^{(m)}\,S^{\Omega(n)}g
					-
					\left(
					\frac1N\sum_{n=1}^{N}
					T^{ne_1}G_1^{(m)}\,T^{ne_2}G_2^{(m)}
					\right)
					\E_{\mu}(g\mid\mathcal I(S))
					\to0
				\end{equation}
				in \(L^2(\mu)\) as $N\to\infty$. Again, since $ \|G_i^{(m)}\|_{L^{\infty}(\mu)}\leq 1$, we can replace $G_{i}^{m}$ by $f_{i,0}$ in \eqref{eq7-1}. Lastly, we replace 
				$f_{i,0}$ by $f_i$. This finishes the proof.
			\end{proof}
			
			

			\section{Further discussions and directions}
			\label{sec:further-questions}
			In this section, some further discussions and directions are presented.
			\subsection{A general principle}
			The proof separates two arithmetic inputs from the dynamical arguments.  The
			first is some kind of synchronization under prime dilations, and the second is a
			nilsequence weighted orbit theorem.  This gives the following conditional
			extension of the main result.
			\begin{proposition}
				\label{prop:arithmetic-replacement}
				Let \(\vartheta:\mathbb N\to\mathbb Z\) satisfy the following two
				conditions.
				\begin{enumerate}
					\item There is a set of primes \(\mathcal P_0\) of positive lower relative
					density such that
					$
					\vartheta(pn)=\vartheta(qn)
					$
					for all distinct \(p,q\in\mathcal P_0\) and all \(n\in\mathbb N\).
					\item For every uniquely ergodic system \((Y,R)\) with the unique
					$R$-invariant Borel probability measure \(\nu\), all \(F\in C(Y)\), all
					\(y\in Y\), and all nilsequence \({\bf b}=(b_n)_{n\geq1}\),
					\[
					\lim_{N\to\infty}\frac1N\sum_{n=1}^N b_n F(R^{\vartheta(n)}y)
					=
					\left(\lim_{N\to\infty}\frac1N\sum_{n=1}^N b_n\right)
					\int_Y F\,d\nu .
					\]
				\end{enumerate}
				Then all our main results remain valid with every occurrence of
				\(\Omega\) replaced by \(\vartheta\).\footnote{Hypothesis (2) is substantially stronger than an unweighted uniquely
					ergodic theorem.  Recent work proves the unweighted conclusion for a broad
					Gaussian distribution class containing both \(\Omega\) and \(\omega\)
					\cite{BergelsonReillyRichter2026}. }
			\end{proposition}
			
			\begin{proof}
				Re-tracing the proof process of \Cref{thm:main-decoupling} gives its $\vartheta$-analogue. Likely, re-tracing the proofs of \Cref{cor:fixed-powers-coefficients}, \Cref{cor:fixed-directions}, and 
				\Cref{cor:divisible-common-differences} and replacing \Cref{thm:main-decoupling} by its $\vartheta$-analogue gives the proposition.
			\end{proof}
			
			\subsection{Further directions I}
			Here, we present three questions on common differences and dimensions. The first one is the Higher-dimensional version of  \Cref{thm:main-corners}.
			\begin{question}[see {\cite[Question~7.1]{Xiao2026}}]
				\label{ques:higher-dimensional-linear-corners}
				Let \(k\geq3\).  If \(A\subset\mathbb N^{k+1}\) has positive upper Banach
				density, does the set of \(d\in\mathbb N\) for which
				\[
				d^*\left(
				A\cap\bigcap_{j=1}^k(A-d e_j)
				\cap(A-\Omega(d)e_{k+1})
				\right)>0
				\]
				have positive upper Banach density?
			\end{question}
			
			There are two separate obstacles.  A prime correlation of \(k\) common linear
			iterates produces, after one coordinate is translated to the origin,
			\(2k-1\) non-zero directions in a \(\mathbb Z^k\)-action.  The
			three-direction characteristic-factor theorem used in the present paper does
			not cover this family.  Moreover, the isotropy coding is genuinely
			two-dimensional.  If \(0\neq w\in\mathbb Z^2\), then
			\(\mathbb Z^2/\mathbb Zw\) is virtually cyclic, so a transverse vector and a
			finite residue decomposition reduce the quotient action to one
			transformation.  For \(w\in\mathbb Z^k\), \(k\geq3\), the quotient has rank
			\(k-1\).  One would therefore need a multi-action weighted decoupling theorem
			in addition to new characteristic factors.  
			
			The second one is a polynomial extension of \Cref{thm:main-corners}.
			\begin{question}
				\label{ques:common-monomial}
				Fix \(r\geq2\).  If \(A\subset\mathbb N^3\) has positive upper Banach
				density, does the set of \(d\in\mathbb N\) for which
				\[
				d^*\bigl(A\cap(A-d^r e_1)\cap(A-d^r e_2)
				\cap(A-\Omega(d)e_3)\bigr)>0
				\]
				have positive upper Banach density?  
				More generally, what is the corresponding
				statement with \(d^r\) replaced by an integer-valued polynomial \(P(d)\) satisfying
				\(P(0)=0\) ?
			\end{question}
			
			Recent work gives seminorm control and popular common-difference results for
			equal or pairwise dependent polynomial corners
			\cite{FrantzikinakisKucaDependent,FrantzikinakisKucaSparse}.  In the mixed
			problem, however, prime dilation creates the polynomial families
			\(P(pn)\) and \(P(qn)\).  The present proof may require a polynomial
			characteristic factor reduction and a polynomial analogue of cyclic tower action.  Even after the isotropy factors are coded,
			the resulting iterates of the transverse action are polynomial rather than
			affine, so the existing proof engine does not apply directly.
			
			The third one is a Hardy sequence extension of \Cref{thm:main-corners}.
			\begin{question}
				\label{ques:hardy-corners}
				Let \(h\) be a Hardy field function of polynomial growth such that
				\[
				\lim_{t\to\infty}\frac{|h(t)-cP(t)|}{\log t}=\infty
				\qquad\text{for every }c\in\mathbb R\text{ and }P\in\mathbb Q[t].
				\]
				If \(A\subset\mathbb N^3\) has positive upper Banach density, does the set of
				\(d\in\mathbb N\) for which
				\[
				d^*\bigl(A\cap(A-\lfloor h(d)\rfloor e_1)
				\cap(A-\lfloor h(d)\rfloor e_2)
				\cap(A-\Omega(d)e_3)\bigr)>0
				\]
				have positive upper Banach density?
			\end{question}
			
			Recent work establishes unweighted common Hardy sequences corner results for broad
			non-polynomial classes, building on the earlier theory of Hardy field
			recurrence \cite{FrantzikinakisHardy2010,FrantzikinakisKucaSparse}.  In the
			mixed case, the prime correlations contain both
			\(\lfloor h(pn)\rfloor\) and \(\lfloor h(qn)\rfloor\); these do not form a
			fixed finite family of affine or polynomial directions in general.  A solution
			would require some uniform Hardy sequences seminorm estimates and a structured weighted
			theorem compatible with the \(\Omega\)-orbit.
			
			\subsection{Further directions II}
			Here, we present the following question on replacements of $\Omega(n)$.
			\begin{question}
				\label{ques:other-additive-functions}
				
				\begin{enumerate}
					\item Does \Cref{thm:main-decoupling} remain
					valid with \(\Omega(n)\) replaced by \(\omega(n)\), the number of distinct
					prime factors of \(n\)?
					\item More generally, let \(\alpha:\mathbb N\to\mathbb Z\) be completely
					additive. 
					Under what general hypotheses on \(\alpha\), does
					\[
					\frac1N\sum_{n=1}^N
					T_1^nf_1\,T_2^nf_2\,
					S^{\alpha(n)}g
					\]
					converge in \(L^2(\mu)\)?  
				\end{enumerate}
			\end{question}

			If \(\alpha(p)\) is constant on a positive relative density subset of primes,
			then by
			\Cref{prop:arithmetic-replacement}, in substantive remaining tasks, the first crucial work is a nilsequences weighted orbit theorem. 
			
			For \(\omega\), exact synchronization is replaced by
			$
			\omega(pn)-\omega(qn)=\one_{q\mid n}-\one_{p\mid n}.
			$
			The exceptional set has upper density at most \(1/p+1/q\).  This seems to suggest a
			quantitative prime dilation argument in which the auxiliary primes tend to
			infinity. Such an argument would still have to be combined with the required
			nilsequences weighted orbit theorem. (The unweighted uniquely ergodic theorem
			for \(\omega\) is known \cite{BergelsonRichter}.)
			
			
			\appendix
			\section{Lifting orbits on nilmanifolds}\label{app:covering-presentation}
			Only linear orbits are needed in this paper.  \Cref{thm:ezk-original} is stated for
			polynomial nilsequences. So, we have to represent such linear orbits as polynomial
			sequences on an Eisner--Zorin-Kranich admissible presentation.  
			
			First, we introduce a covering property of compactly generated nilpotent Lie groups.
			\begin{lemma}[Cf. {\cite[Proposition B.1 and Lemma B.2]{KanigowskiLemanczykRichterTeravainen2024}}]
				\label{lem:simply-connected-algebraic-presentation}
				Let \(X=G/\Gamma\), where \(G\) is a compactly generated nilpotent Lie group
				and \(\Gamma\leq G\) is discrete and co-compact. Then there are a compactly
				generated nilpotent Lie group \(\bar G\), a discrete and co-compact 
				\(\bar\Gamma\leq\bar G\), and a surjective covering homomorphism
				\(\rho:\bar G\to G\) such that the connect component of the identity
				$
				\bar G^\circ \text{ is simply connected},\ 
				\bar\Gamma=\rho^{-1}(\Gamma).
				$
				Moreover, the induced map
				$
				\bar\rho:\bar G/\bar\Gamma\to G/\Gamma,
				\bar g\bar\Gamma\mapsto\rho(\bar g)\Gamma
				$
				is a diffeomorphism and for all $\bar a,\bar g\in\bar G$,
				$
				\bar\rho(\bar a\bar g\bar\Gamma)
				=\rho(\bar a)\bar\rho(\bar g\bar\Gamma).
				$
			\end{lemma}
			Next, we apply the above lemma to lift linear orbits on nilmanifolds.
			\begin{lemma}
				\label{lem:malcev-orbit-replacement}
				Let \(M=N/\Delta\) be a nilmanifold, where \(N\) is a compactly
				generated nilpotent Lie group and \(\Delta\leq N\) is discrete and co-compact.
				Then there are a compactly generated nilpotent Lie group \(\widehat N\) with the simply connected \(\widehat N^\circ\), a
				discrete and co-compact \(\widehat\Delta\leq\widehat N\), a surjective covering homomorphism
				$
				\vartheta:\widehat N\longrightarrow N\ \text{with}\ 
				\widehat\Delta=\vartheta^{-1}(\Delta),
				$
				a \(\widehat\Delta\)-rational filtration \(\widehat N_\bullet\), and an
				Mal'cev basis \(\widehat{\mathcal M}\) for $\widehat N^\circ/(\widehat N^\circ\cap \widehat\Delta)$ adapted to \({\widehat N}^{\circ}_{\bullet}\) such that
				the induced map
				$
				p:\widehat N/\widehat\Delta\longrightarrow M,
				\widehat x\widehat\Delta\mapsto\vartheta(\widehat x)\Delta
				$
				is a diffeomorphism satisfying
				\begin{equation}\label{eq:linear-model-equivariance}
					p(\widehat a\widehat x\widehat\Delta)
					=\vartheta(\widehat a)p(\widehat x\widehat\Delta)\ \text{for all}\ \widehat a,\widehat x\in\widehat N,
				\end{equation}
				and the filtered nilmanifold
				$(\widehat N/\widehat\Delta,\widehat N_{\bullet})$ is Eisner--Zorin-Kranich admissible.  Moreover, we also have the following:
				\begin{itemize}
					\item 	For every
					\(K\geq0\), there is \(C_{M,K}>0\) such that
					\begin{equation}\label{eq:malcev-Ck-pullback}
						\|F\circ p\|_{C^K(\widehat N/\widehat\Delta)}
						\leq C_{M,K}\|F\|_{C^K(M)}\ \text{for all}\ F\in C^\infty(M).
					\end{equation}
					\item For every \(a\in N\) and \(y\in M\), there is
					\(\widehat g\in P(\mathbb Z,\widehat N_\bullet)\) such that for all $n\in \Z$,
					\begin{equation}\label{eq:malcev-linear-orbit-lift}
						p\bigl(\widehat g(n)\widehat\Delta\bigr)=a^ny.
					\end{equation}
				\end{itemize}
			\end{lemma}
			
			\begin{proof}
				Apply \Cref{lem:simply-connected-algebraic-presentation} to \(N/\Delta\).
				This gives the required \(\widehat N,\widehat\Delta,\vartheta\), and \(p\), with
				\eqref{eq:linear-model-equivariance}. Take the lower central series
				filtration \(\widehat N_\bullet\) by letting
				$
				\widehat N_0=\widehat N_1=\widehat N;\ \text{for}\ j\geq1, 
				\widehat N_{j+1}=[\widehat N,\widehat N_j].
				$
				It has finite length because \(\widehat N\) is nilpotent. 
				By \cite[the paragraph after Definition~2.16]{EisnerZorinKranich}, 
				$\widehat N_\bullet$ is
				\(\widehat\Delta\)-rational and 
				we can fix a Mal'cev basis for $\widehat N^\circ/(\widehat N^\circ\cap \widehat\Delta)$ adapted to \({\widehat N}^{\circ}_{\bullet}\).  This gives an Eisner--Zorin-Kranich admissible representation.
				
				The estimate \eqref{eq:malcev-Ck-pullback} follows from the chain rule for the
				fixed smooth map \(p\) between compact manifolds directly.
				
				Fix \(a\in N\) and \(y\in M\).  Choose
				$
				\widehat a\in\vartheta^{-1}(a),
				\widehat y\in\widehat N,
				p(\widehat y\widehat\Delta)=y,
				$
				and put
				$
				\widehat g(n)=\widehat a^{\,n}\widehat y.
				$
				Note that
				\[
				D_h\widehat g(n)
				=\widehat g(n)^{-1}\widehat g(n+h)
				=\widehat y^{-1}\widehat a^{\,h}\widehat y
				\in\widehat N_1.
				\]
				This gives
				\(\widehat g\in P(\mathbb Z,\widehat N_\bullet)\).  Finally,
				\eqref{eq:linear-model-equivariance} gives
				$
				p(\widehat g(n)\widehat\Delta)
				=(\vartheta(\widehat a))^n p(\widehat y\widehat\Delta)
				=a^ny.
				$ This finishes the proof.
			\end{proof}
			\section{Finite Eisner--Zorin-Kranich representations}\label{appB}
			In this section, we focus on finite Eisner--Zorin-Kranich representations of bounded sequences from product spaces. First, we introduce the notion of finite Eisner--Zorin-Kranich representations.
			\begin{definition}
				\label{def:finite-ezk-representation}
				Let \((X,\mathcal X,\mu,U)\) be a $\Z$-system, and
				let \(c:X\times\mathbb Z\to\mathbb C\) be bounded and measurable.  Fix
				\(s\geq1\), \(K\geq0\), and \(B>0\). \(c\) is said to have a
				\emph{finite Eisner--Zorin-Kranich representation with parameters \((s,K,B)\)} if there are
				an integer \(J\geq1\), a co-null set \(X_0\subseteq X\), and a finite family
				$
				\mathscr N=
				\{(G_j/\Gamma_j,G_{j,\bullet},\mathcal M_j):1\leq j\leq J\}
				$
				of Eisner--Zorin-Kranich admissible filtered nilmanifolds of degree at most \(s\), such that the following hold:
				\begin{itemize}
					\item 	for all \(x\in X_0,n\in \Z\),
					\begin{equation}\label{eq:finite-ezk-representation}
						c(x,n)=\sum_{j=1}^{J}F_{j,x}(g_{j,x}(n)\Gamma_j),
					\end{equation}
					where
					$
					g_{j,x}\in P(\mathbb Z,G_{j,\bullet}),
					F_{j,x}\in C^\infty(G_j/\Gamma_j),
					$
					and
					\begin{equation}\label{eq:finite-ezk-uniform-Ck}
						\sup_{x\in X_0}\sup_{1\leq j\leq J}
						\|F_{j,x}\|_{C^K(G_j/\Gamma_j)}\leq B .
					\end{equation}
					\item $\mathscr N$ satisfies 
					\begin{equation}\label{eq:finite-ezk-order-condition}
						K\geq \max_{1\leq j\leq J}\kappa_j(s),\ \text{where}\ 
						\kappa_j(s)=
						\sum_{r=1}^{s}(d_{j,r}-d_{j,r+1})\binom{s}{r-1},
						d_{j,r}=\dim G_{j,r}.
					\end{equation}
				\end{itemize}
			\end{definition}
			The above definition can induce the following notion.
			\begin{definition}
				\label{def:finite-ezk-stability}
				Let \((X,\mathcal X,\mu,U)\) be a $\Z$-system.  A
				bounded measurable \(c:X\times\mathbb Z\to\mathbb C\) is
				\emph{Eisner--Zorin-Kranich regular with respect to \(U\)} (or \emph{finitely Eisner--Zorin-Kranich stable})
				if the following condition holds.  For every \(R\geq1\), every
				\(\boldsymbol\epsilon=(\epsilon_1,\ldots,\epsilon_R)\in\{0,1\}^R\), and every
				finite list
				$
				(r_\nu,t_\nu,a_\nu)\in\mathbb Z^3,
				1\leq\nu\leq R,
				$
				there exist parameters \((s,K,B)\) 
				and a finite family
				$
				\mathscr N
				$
				of Eisner--Zorin-Kranich admissible filtered nilmanifolds of degree at most \(s\)
				such that for every \(\mathbf h=(h_1,\ldots,h_R)\in\mathbb Z^R\), 
				\begin{equation}\label{eq:derived-ezk-stable-weight}
					c_{\mathbf h}^{\boldsymbol\epsilon}(x,n)
					:=
					\prod_{\nu=1}^{R}
					\mathcal C^{\epsilon_\nu}
					c(U^{r_\nu n+t_\nu}x,a_\nu n+h_\nu)
				\end{equation}
				has a finite Eisner--Zorin-Kranich representation with parameters \((s,K,B)\) and in this representation, all Eisner--Zorin-Kranich admissible filtered nilmanifolds are from $\mathscr N$.  
			\end{definition}

			Next, we give an estimate on weighted ergodic averages, where the weight has a finite Eisner--Zorin-Kranich representation.
			\begin{lemma}
				\label{lem:one-slope}
				Let \((X,\mathcal X,\mu,U)\) be a $\Z$-system,
				\(a\in\mathbb Z\setminus\{0\}\), and \(s\geq1\).  Let
				\(c:X\times\mathbb Z\to\mathbb C\) be bounded and measurable.  Assume that
				\(c\) has a finite Eisner--Zorin-Kranich representation with parameters \((s,K,B)\) in the
				sense of \Cref{def:finite-ezk-representation}.  Then, for every
				\(H\in L^\infty(\mu)\) with \(\|H\|_{L^\infty(\mu)}\leq1\),
				\[
				\limsup_{N\to\infty}
				\left\|
				\frac1N\sum_{n=1}^{N}c(x,n)U^{an}H
				\right\|_{L^2(\mu)}
				\leq
				C\|H\|_{U^{s+1}(X,\mu,U)},
				\]
				where \(C\) depends only on \(a\) and the finite Eisner--Zorin-Kranich representation of \(c\).
			\end{lemma}
			
			\begin{proof}
				By the finite Eisner--Zorin-Kranich representation of \(c\), we get a co-null set $X_0$ and the family
				$
				\mathscr N=\{(G_\ell/\Gamma_\ell,G_{\ell,\bullet},\mathcal M_\ell):
				1\leq\ell\leq L\}
				$.
				And, for every \(x\in X_0,n\in \Z\),
				\[
				c(x,n)=\sum_{\ell=1}^{L}F_{\ell,x}(g_{\ell,x}(n)\Gamma_\ell),
				\]
				with \(g_{\ell,x}\in P(\mathbb Z,G_{\ell,\bullet})\), and with the
				\(C^K\)-norms of all \(F_{\ell,x}\) bounded by $B$.  
				
				Put \(T=U^a\), and disintegrate \(\mu\) into ergodic components for \(T\) by
				\[
				\mu=\int_\Omega \mu_\omega\,d\lambda(\omega).
				\]
				Set
				\[
				A_N(x)=\frac1N\sum_{n=1}^{N}c(x,n)H(T^n x).
				\]
				For \(\lambda\)-a.e. \(\omega\in \Omega\), for \(\mu_\omega\)-a.e. \(x\in X\), \(x\) is fully generic for \(H\) in
				\((X,\mathcal X,\mu_\omega,T)\).  Hence,
				\Cref{cor:ezk-finite-Ck-family} gives, for such \(\omega\) and \(x\),
				\[
				\limsup_{N\to\infty}|A_N(x)|
				\leq
				C_0\|H\|_{U^{s+1}(X,\mu_\omega,T)},
				\]
				where \(C_0\) is independent of \(\omega,x\), and \(H\).
				By the Fatou Lemma,
				for \(\lambda\)-a.e. \(\omega\in \Omega\),
				\[
				\limsup_{N\to\infty}\|A_N\|_{L^2(\mu_\omega)}^2
				\leq
				C_0^2\|H\|_{U^{s+1}(X,\mu_\omega,T)}^2 .
				\]
				Applying the Fatou Lemma again yields
				\[
				\limsup_{N\to\infty}\|A_N\|_{L^2(\mu)}^2
				\leq
				C_0^2\int_\Omega
				\|H\|_{U^{s+1}(X,\mu_\omega,T)}^2\,d\lambda(\omega).
				\]
				Note that 
				\[
				\left(\int_\Omega
				\|H\|_{U^{s+1}(X,\mu_\omega,T)}^2\,d\lambda\right)^{1/2}
				\leq
				\left(\int_\Omega
				\|H\|_{U^{s+1}(X,\mu_\omega,T)}^{2^{s+1}}\,d\lambda\right)^{1/2^{s+1}} .
				\]
				By \Cref{lem:HK-seminorm-ergodic-decomposition}, the last expression is
				\(\|H\|_{U^{s+1}(X,\mu,U^a)}\).  Finally,
				\Cref{lem:one-dimensional-power-comparison} gives
				\[
				\|H\|_{U^{s+1}(X,\mu,U^a)}
				\lesssim_{a,s}
				\|H\|_{U^{s+1}(X,\mu,U)} .
				\]
				This proves the lemma.
			\end{proof}
			%
			%

		\end{document}